\documentclass[11pt, notitlepage]{article}
\usepackage{amssymb,amsmath,comment}
\catcode`\@=11 \@addtoreset{equation}{section}
\def\thesection{\arabic{section}}

\def\theequation{\thesection.\arabic{equation}}
\catcode`\@=12
\usepackage{colortbl}%
\usepackage{a4wide}

\newcommand{\ds} {\displaystyle}
\newcommand{\e}{\epsilon}

\newcommand{\al} {\alpha}
\newcommand{\ba} {\beta}
\newcommand{\de} {\delta}
\newcommand{\ga} {\gamma}

\newcommand{\Om} {\Omega}
\newcommand{\ra} {\rightarrow}

\newcommand{\De} {\Delta}

\newcommand{\noi} {\noindent}

\newcommand{\mb} {\mathbb}
\newcommand{\mc} {\mathcal}

\usepackage[all]{xy}
\catcode`\@=11
\def\theequation{\@arabic{\c@section}.\@arabic{\c@equation}}
\catcode`\@=12

\def\proof{\noindent{\textbf{Proof. }}}
\def\QED{\hfill {$\square$}\goodbreak \medskip}

\newtheorem{Theorem}{Theorem}[section]
\newtheorem{Lemma}[Theorem]{Lemma}
\newtheorem{Proposition}[Theorem]{Proposition}

\newtheorem{Remark}[Theorem]{Remark}
\newtheorem{Definition}[Theorem]{Definition}
\newtheorem{Example}{Example}

\begin{document}

\title{Kirchhoff equations with Choquard exponential type nonlinearity involving the fractional Laplacian}
\author{
{\bf   \; Sarika Goyal\footnote{email: sarika.goyal@bennett.edu.in}}\\
{\small Department of Mathematics, }\\
{\small Bennett University Greater Noida, Uttar Pradesh-201310, India},\;\\
 {\bf Tuhina Mukherjee\footnote{email: tuhina@tifrbng.res.in,}}\\
{\small T.I.F.R.  Centre for Applicable Mathematics,}\\
{\small Post Bag No. 6503, Sharadanagar, Yelahanka New Town, Bangalore 560065. }}
\date{}

\maketitle
%SIAM CONFERENCE ABSTRACT
%This talk deals with the existence of non-negative solutions of fractional-Kirchhoff equation with exponential nonlinearity of choquard type \[ - M\left(\int\int_{R^{2n}} \frac{|u(x)-u(y)|^{\frac{n}{s}}}{|x-y|^{2n}}\right) (-\Delta)^{s}_{n/s} u = \left(\int_{\Omega}\frac{G(y,u)}{|x-y|^{\mu}}\right)g(x,u) \; \text{in}\; \Omega,\\ \quad \quad u =0\quad\text{on} \quad \mathbb R^n \setminus \Omega, \] where $(-\Delta)^{s}_{n/s}$ is the $n/s$-fractional Laplace operator, $n\geq 1$, $s\in(0,1)$ $\Omega\subset \mathbb R^n$ is a bounded domain with Lipschitz boundary, $M:\mathbb R^+\rightarrow \mathbb R^+$ and $g:\Omega\times\mathbb R\rightarrow \mathbb R$ are continuous functions, where $g$ behaves like $e^{|u|^{\frac{n}{n-s}}}$ as $|u|\rightarrow\infty$ and $G$ is primitive of $g$.

\begin{abstract}
\noi In this article, we deal with the existence of
non-negative solutions of the class of following  non local problem
$$  \left\{
\begin{array}{lr}
 \quad - M\left(\displaystyle\int_{\mb R^n}\int_{\mb R^{n}} \frac{|u(x)-u(y)|^{\frac{n}{s}}}{|x-y|^{2n}}~dxdy\right) (-\De)^{s}_{n/s} u=\left(\displaystyle\int_{\Om}\frac{G(y,u)}{|x-y|^{\mu}}~dy\right)g(x,u) \; \text{in}\;
\Om,\\
\quad \quad u =0\quad\text{in} \quad \mb R^n \setminus \Om,
\end{array}
\right.
$$
where $(-\De)^{s}_{n/s}$ is the $n/s$-fractional Laplace operator, $n\geq 1$, $s\in(0,1)$ such that $n/s\geq 2$, $\Om\subset \mb R^n$ is a bounded domain with Lipschitz boundary, $M:\mb R^+\rightarrow \mb R^+$ and $g:\Omega\times\mb R\rightarrow \mb R$ are continuous functions, where $g$ behaves like $\exp({|u|^{\frac{n}{n-s}}})$ as $|u|\ra\infty$.
\medskip

\noi \textbf{Key words:} Doubly non local problems, Kirchhoff equation, Choquard nonlinearity, Trudinger-Moser nonlinearity.

\medskip

\noi \textit{2010 Mathematics Subject Classification:} 35R11, 35J60,
35A15

\end{abstract}

\bigskip
%\vfill\eject

\section{Introduction}
\setcounter{equation}{0}
Let $n\geq 1$, $s\in(0,1)$ such that $n/s \geq 2$ and  $\Om\subset \mb R^n$ be a bounded domain with Lipschitz boundary then we intend to study the existence of a non negative solutions of following fractional Kirchhoff type problem with Trudinger-Moser type Choquard nonlinearity
$$ \mc{( M)}\quad \left\{
\begin{array}{lr}
 \quad - M\left(\displaystyle\int_{\mb R^n}\int_{\mb R^{n}} \frac{|u(x)-u(y)|^{\frac{n}{s}}}{|x-y|^{2n}}~dxdy\right) (-\De)^{s}_{n/s} u=\left(\displaystyle\int_{\Om}\frac{G(y,u)}{|x-y|^{\mu}}~dy\right)g(x,u) \; \text{in}\;
\Om,\\
\quad \quad u =0\quad\text{in} \quad \mb R^n \setminus \Om,
\end{array}
\right.
$$
where $(-\De)^{s}_{n/s}$ is the $n/s$-fractional Laplace operator which, up to a normalizing constant, is defined as
\begin{align*}
(-\De)^{s}_{n/s} u(x) = 2 \lim_{\e \ra 0^+} \int_{\mb R^n\setminus B_{\e}(x)} \frac{|u(x)-u(y)|^{\frac{n}{s}-2}(u(x)-u(y))}{|x-y|^{2n}} dy, \quad x\in \mb R^n, \; u\in C_{0}^{\infty}(\mb R^n).
\end{align*}
The functions $M:\mb R^+\rightarrow \mb R^+$ and $g:\Omega\times\mb R\rightarrow \mb R$ are continuous satisfying some appropriate conditions which will be stated later on.\\
Our problem $(\mc M)$ is basically driven by the Hardy-Littlewood-Sobolev inequality and the Trudinger-Moser inequality. Let us first recall the following well known Hardy-Littlewood-Sobolev inequality [Theorem 4.3, p.106] \cite{lieb}.
 \begin{Proposition}\label{HLS}
(\textbf {Hardy-Littlewood-Sobolev inequality}) Let $t$, $r>1$ and $0<\mu<n $ with $1/t+\mu/n+1/r=2$, $g \in L^t(\mathbb R^n)$ and $h \in L^r(\mathbb R^n)$. Then there exists a sharp constant $C(t,n,\mu,r)$, independent of $g,$ $h$ such that
 \begin{equation}\label{HLSineq}
 \int_{\mb R^n}\int_{\mb R^n} \frac{g(x)h(y)}{|x-y|^{\mu}}\mathrm{d}x\mathrm{d}y \leq C(t,n,\mu,r)\|g\|_{L^t(\mb R^n)}\|h\|_{L^r(\mb R^n)}.
 \end{equation}
{ If $t =r = \textstyle\frac{2n}{2n-\mu}$ then
 \[C(t,n,\mu,r)= C(n,\mu)= \pi^{\frac{\mu}{2}} \frac{\Gamma\left(\frac{n}{2}-\frac{\mu}{2}\right)}{\Gamma\left(n-\frac{\mu}{2}\right)} \left\{ \frac{\Gamma\left(\frac{n}{2}\right)}{\Gamma(n)} \right\}^{-1+\frac{\mu}{n}}.  \]
 In this case there is equality in \eqref{HLSineq} if and only if $g\equiv (constant)h$ and
 \[h(x)= A(\gamma^2+ |x-a|^2)^{\frac{-(2n-\mu)}{2}}\]
 for some $A \in \mathbb C$, $0 \neq \gamma \in \mathbb R$ and $a \in \mathbb R^n$.}
 \end{Proposition}
The study of Choquard equations originates from the work of S. Pekar \cite{pekar} and P. Choquard \cite{choquard} where they used elliptic equations with Hardy-Littlewood-Sobolev type nonlinearity to describe the quantum theory of a polaron at rest and to model an electron trapped in its own hole in the Hartree-Fock theory, respectively. For more details on the application of Choquard equations, we refer \cite{survey}. On the other hand, the boundary value problems involving Kirchhoff equations arise in several physical and
biological systems. These type of non-local problems were initially observed by Kirchhoff in 1883
 in the study of string or membrane vibrations to describe the transversal oscillations of a stretched string, particularly, taking into account the subsequent change in string length caused by oscillations.

L\"{u} \cite{Lu} in $2015$ studied the following Kirchhoff problem with Choquard nonlinearity
\[-\left( a+ b \int_{\mb R^3}|\nabla u|^2~dx\right)\Delta u +(1+\mu g(x))u = (|x|^{-\alpha}\ast |u|^p)u|^{p-2}u\; \text{in}\; \mb R^3 \]
for $a>0,\; b\geq 0,\; \alpha \in (0,3),\; p \in (2,6-\alpha)$, $\mu>0$ is a parameter and $g$ is a nonnegative continuous potential with some growth assumptions. He proved the existence of solution to the above problem for $\mu$ sufficiently large and also {showed} their concentration behavior when $\mu$ approaches $+\infty$. In \cite{FCX}, authors discuss the existence and  concentration of sign-changing solutions to a class of Kirchhoff-type systems with Hartree-type nonlinearity in $\mb R^3$ by the minimization argument on the sign-changing Nehari manifold and a quantitative deformation lemma. In the nonlocal case {that is problems involving the fractional Laplace operator}, Kirchhoff problem with Choquard nonlinearity has been studied by Pucci et al. in \cite{pucci} via variational techniques.\\

The study of elliptic equations involving nonlinearity with exponential growth are motivated by
the following Trudinger-Moser inequality in \cite{martinazi}, namely
\begin{Theorem} \label{moser}
let $\Om$ be a open bounded domain then we define $\tilde{W}^{s,n/s}_{0}(\Om)$ as the completion of $C_{c}^{\infty}(\Om)$ with respect to the norm $\|u\|^{\frac{n}{s}}= \displaystyle\int_{\mb R^{n}}\int_{\mb R^n} \frac{|u(x)-u(y)|^{\frac{n}{s}}}{|x-y|^{2n}} dx dy$. Then there exists a positive constant $\alpha_{n,s}$ given by
\[\alpha_{n,s}= \frac{n}{\omega_{n-1}}\left( \frac{\Gamma(\frac{n-s}{2})}{\Gamma(s/2)2^s \pi^{n/2}}\right)^{-\frac{n}{n-s}},\]
where $\omega_{n-1}$ be the surface area of the unit sphere in $\mb R^n$ and $C_{n,s}$ depending only on $n$ and $s$ such that
\begin{equation}\label{TM-ineq}
    \sup_{u \in \tilde{W}^{s,n/s}_{0}(\Om),\; \|u\|\leq 1}\int_\Om \exp\left( \alpha |u|^{\frac{n}{n-s}}\right)~dx\leq C_{n,s}|\Om|
    \end{equation}
for each $\alpha \in [0,\alpha_{n,s}]$.
Moreover there exists a $ \alpha_{n,s}^* \geq \alpha_{n,s}$ such that {the right hand side of\eqref{TM-ineq} is $+\infty$} for $\alpha>\alpha_{n,s}^*$.
\end{Theorem}
It is proved in \cite{ruf} (see Proposition 5.2) that
\[\alpha_{n,s}^*= n \left(\frac{2(n\mc W_n)^2 \Gamma(\frac{n}{s}+1)}{n!}\sum_{i=0}^{\infty} \frac{(n+i-1)!}{i! (n+2i)^{\frac{n}{s}}}\right)^{\frac{s}{n-s}},\]
where $\mc W_n= \frac{w_{n-1}}{n}$ {is} the volume of the unit sphere in $\mb R^n$. It {is still} unknown whether $ \alpha_{n,s}^* = \alpha_{n,s}$ {or not}.

The $p$-fractional Kirchhoff problems involving the Trudinger-Moser type nonlinearity has been recently addressed in \cite{MRZ1, MRZ2}. We also refer \cite{pawan1, pawan2} to the readers, in the linear case i.e. when $p=2$. The Choquard equations with exponential type nonlinearities has been comparatively less attended. In this regard, we cite \cite{YangJDE} where authors studied a singularly perturbed nonlocal Schr\"odinger equation via variational techniques. We also refer \cite{JCA} for reference. On a similar note, there is no literature available on Kirchhoff problems involving the Choquard exponential nonlinearity except the very recent article \cite{AGMS} where authors studied the existence  of positive solutions to the following problem
\[-m\left(\int_\Om |\nabla u|^n ~dx \right)\Delta_nu = \left(\int_\Om \frac{F(y,u)}{|x-y|^{\mu}}dy\right)f(x,u), \; u>0\;\text{in}\; \Om,\; u=0\;\text{in}\; \partial\Om\]
where $-\Delta_n = \nabla. (|\nabla u|^{n-2}\nabla u)$, $\mu\in (0,n)$, $n\geq 2$, $m$ and $f$ are continuous functions satisfying some additional assumptions, using the concentration compactness arguments. They also {established} multiplicity result corresponding to a perturbed problem via minimization over suitable subsets of Nehari manifold. Whereas in the {$p$-fractional laplacian} case, motivated by above research, our paper represents the first article to consider the Kirchhoff problem with Choquard exponential nonlinearity. \\

The problem of the type $(\mc M)$ are categorized under doubly nonlocal problems because of the presence of the term $M\left(\displaystyle \int_{\mb R^{n}}\int_{\mb R^n} \frac{|u(x)-u(y)|^{\frac{n}{s}}}{|x-y|^{2n}} dx dy\right)$ and {$\displaystyle\left(\int_{\Om}\frac{G(y,u)}{|x-y|^{\mu}}~dy\right)g(x,u)$}
which does not allow the problem $(\mc M)$ to be  a pointwise identity. Additionally, we also deal with the degenerate case of Kirchhoff problem which is a new result even in the case of $s=1$. This phenomenon arises mathematical difficulties which makes the study of such a class of problem interesting. Generally, the main difficulty encountered in Kirchhoff problems is the competition between the growths of $M$ and $g$.
Precisely, mere weak limit of a Palais Smale (PS) sequence is not enough to claim that it is a weak solution to $(\mc M)$ because of presence of the function $M$,  which holds in the case of $M \equiv 1$. Next technical hardship emerge while proving convergence of the Choquard term with respect to (PS) sequence. We use delicate ideas in Lemma \ref{wk-sol} and Lemma \ref{PS-ws} to establish it. Following a variational approach, we prove that the corresponding energy functional to $(\mc M)$ satisfies the Mountain pass geometry and the Mountain pass critical level stays below a threshold (see Lemma \ref{lem7.2}) using the Moser type functions established by Parini and Ruf in \cite{ruf}. Then we perform a convergence analysis of the Choquard term with respect to the (PS)-sequences in Lemma  \ref{wk-sol}. This along with the higher integrability Lemma \ref{plc} benefited us to get the weak limit of (PS)-sequence as a weak solution of $(\mc M)$ leading to build the proof of our main result. The approach although may not be completely new but the problem is comprehensively afresh. \\

Our article is divided into 3 sections- Section $2$ illustrates the functional set up to study $(\mc M)$ and contains the main result that we intend to establish. Section $3$ contains the proof of our main result.

\section{Functional Setting and Main result}
 Let us consider the usual fractional Sobolev space
\[W^{s,p}(\Om):= \left\{u\in L^{p}(\Om); \frac{(u(x)-u(y))}{|x-y|^{\frac{n}{p}+s}}\in L^{p}(\Om\times\Om)\right\}\]
 endowed with the norm
\begin{align*}
\|u\|_{W^{s,p}(\Om)}=\|u\|_{L^p(\Om)}+ \left(\int_{\Om}\int_{\Om}
\frac{|u(x)-u(y)|^{p}}{|x-y|^{n+ps}}dxdy \right)^{\frac 1p}
\end{align*}
where $\Om \subset \mb R^n$ is an open set. {We denote $W^{s,p}_0(\Omega)$ as the completion of the space $C_c^\infty(\Omega)$ with respect to the norm $\|\cdot\|_{W^{s,p}(\Omega)}$.} To study fractional Sobolev spaces in details we refer to \cite{hitchker}.
Now we define
 \[ X_0 = \{u\in W^{s,n/s}(\mb R^n) : u = 0 \;\text{in}\; \mb R^n\setminus \Om\}\]
with respect to the norm
\[\|u\|_{X _0}=\left( \int_{\mb R^{n}}\int_{\mb R^n}\frac{|u(x)- u(y)|^{\frac{n}{s}}}{|x-y|^{2n}}dx
dy\right)^{\frac sn}= \left(\int_{Q}\frac{|u(x)- u(y)|^{\frac{n}{s}}}{|x-y|^{2n}}dx
dy\right)^{\frac sn},\]
where  $Q=\mb R^{2n}\setminus(\mc C\Om\times \mc C\Om)$ and
 $\mc C\Om := \mb R^n\setminus\Om$. Then $X_0$ is a reflexive Banach space and continuously embedded in $W^{s,p}_0(\Om)$. Also $X_0 \hookrightarrow \hookrightarrow L^q(\Om)$ compactly for each $q \in [1,\infty)$. Note that the norm
$\|.\|_{X_0}$ involves the interaction between $\Om$ and $\mb
R^n\setminus\Om$. We denote $\|.\|_{X_0}$ by $\|.\|$ in future, for notational convenience. This type of functional setting was introduced by Servadei and Valdinoci for $p=2$ in \cite{mp} and for $p\ne 2$ in \cite{ssi}.

\noi Moreover, we define the space
%\[W_{0}^{s,p}(\Om)= \overline{C_{0}(\Om)}^{\|\cdot\|_{W^{s,p}(\Om)}},\]
\[\tilde{W}_{0}^{s,p}(\Om)= \overline{C_{0}(\Om)}^{\|\cdot\|_{W^{s,p}(\mb R^n)}}.\]
The space $\tilde{W}_{0}^{s,p}(\Om)$ is equivalent to the completion of $C_{0}^{\infty}(\Om)$ with respect to the semi norm {$\int_{\mb R^{n}}\int_{\mb R^{n}}\frac{|u(x)- u(y)|^{\frac{n}{s}}}{|x-y|^{2n}}dxdy$} (see for example [\cite{BLP}, Remark 2.5]). If $\partial \Om$ is Lipschitz, then $\tilde{W}_{0}^{s,p}(\Om) =X_0$, (see[\cite{BPS}, Proposition B.1]).
The embedding $W_{0}^{s,\frac{n}{s}}(\Om) \ni u\longmapsto \exp({|u|^{\ba}}) \in L^{1}(\Om)$ is compact for all $\ba\in\left(1,\frac{n}{n-s}\right)$
and is continuous when $\ba=\frac{n}{n-s}$.

\noi We now state our assumptions on $M$ and $g$. The function $M:\mb R^+\rightarrow \mb R^+$ is a continuous function which satisfies the following assumptions:
\begin{enumerate}
  \item[$(M1)$] For all $t$, $s\geq 0$, it holds
  \[ \hat M(t+s)\geq \hat M(t)+\hat M(s),\]
  where $\hat M(t)= \int_{0}^{t} M(s)ds$, the primitive of $M$.
 % \item[$(M2)$] There exist constants $a_1$, $a_2>0$ and $t_0>0$ such that for some $\sigma\in\mb R$
%  \[  M(t)\leq a_1 + a_2 t^{\sigma}, \; \text{for all}\; t\geq t_0.\]
  \item[$(M2)$] There exists a $\gamma >1$ such that $t\mapsto \frac{M(t)}{t^{\gamma-1}}$ is non increasing for each $t>0$.
  \item[$(M3)$] For each $b>0$, there exists a $\kappa:= \kappa(b)>0$ such that $M(t)\geq \kappa$ whenever $t \geq b$.
 \end{enumerate}
The condition $(M3)$ asserts that the function $M$ has possibly a zero only when $t=0$.
\begin{Remark}\label{rem-M}
From $(M2)$, we can easily deduce that $\gamma \hat M(t)-M(t)t$ is non decreasing for $t>0$ and
\begin{equation}\label{cnd-M}
     \gamma \hat M(t)-M(t)t \geq 0\;\;\; \forall \; t \geq 0.
\end{equation}
\end{Remark}
We also have the following remark as a consequence of \eqref{cnd-M}.

\begin{Remark}
For each $t \geq 0$, by using \eqref{cnd-M} we have
\[\frac{d}{dt}\left(\frac{\hat M(t)}{t^\gamma}\right)= \frac{M(t)}{t^\gamma}-\frac{\gamma \hat M(t)}{t^{\gamma +1}}\leq 0.\]
So the map $t \mapsto \frac{\hat M(t)}{t^\gamma}$ is non increasing for $t> 0$. Hence
\begin{equation}\label{cnd-M1}
      \hat M(t)\geq \hat M(1)t^\gamma\; \text{for all}\; t \in[ 0,1],
\end{equation}
and
\begin{equation}\label{cnd-M2}
      \hat M(t)\leq \hat M(1)t^\gamma\; \text{for all}\; t \geq 1.
\end{equation}
\end{Remark}

\noi We note that the condition $(M1)$ is valid whenever $M$ is
non decreasing.
\begin{Example}
Let  $M(t)=m_0+at^{\gamma-1}$, where $m_0,a\geq 0$ and $\gamma>1$ such that  $m_0+a>0$ then $M$ satisfies the
conditions $(M1)-(M3)$. If $m_0=0$, this forms an example of the degenerate case whereas of the non degenerate case if $m_0>0$.
\end{Example}
%%%%%%%%%%%%%%%%%%%%%%%%%%%%%%%%%%%%%%%%%%%%%%%%%%%%%%%%%%%%%%%%%%%
\noi The nonlinearity $g:\Omega\times\mb R\rightarrow \mb R$ is {a continuous function} such that $g(x,t)=h(x,t) \exp({|t|^{\frac{n}{n-s}}})$, where $h(x,t)$ satisfies the following assumptions:
\begin{enumerate}
\item[$(g1)$] $h\in C^1(\overline{\Om}\times \mb R)$, $h(x,t)=0,$ for all $t\le 0$, $h(x,t)>0,$ \text{for all}
$t>0$.
\item[$(g2)$] For any $\e>0,$ $\ds \lim_{t\ra \infty}\sup_{x\in \overline{\Om}} h(x,t) \exp(-\e |t|^{\frac{n}{n-s}} )=0$, $\ds\lim_{t\ra \infty}\inf_{x\in \overline{\Om}} h(x,t) \exp(\e|t|^{\frac{n}{n-s}})=\infty.$
\item[$(g3)$] There exist positive constants $T$, $T_0$ and $\gamma_0$ such that
\[ 0<t^{\gamma_0} G(x,t)\le T_0 g(x,t)\;\mbox{for all}\; (x,t)\in \Om\times[t_0,+\infty).\]
%\item[$(g4)$] For each $\ds x\in \Omega, \frac{g(x,t)}{t^{2\frac{n}{s}-1}}$ is increasing for $t>0$ and $\ds\lim_{t\rightarrow 0^+} \frac{g(x,t)}{t^{2\frac{n}{s}-1}}=0,\;\text{uniformly in }\; x\in \Om.$
%\item[$(g5)$] $\ds \lim_{t\rightarrow \infty} t h(x,t)=\infty.$
\item[$(g4)$] For $\gamma >1$ (defined in (M2)), there exists a $l>\frac{\gamma n}{2s}-1$ such that the map $t \mapsto \frac{g(x,t)}{t^{l}}$ is increasing on $\mb R^{+} \setminus \{0\}$, uniformly in $x\in \Om$.
\end{enumerate}
%%%%%%%%%%%%%%%%%%%%%%%%%%%%%%%%%%%%%%%%%%%%%%%%%%%%%%%%%%%%%%%%%%%%%%%%%%%%%%%%%%%%%%%%%%%%%5
\begin{Remark}\label{rem2}Condition $(g4)$ implies that for each $\ds x\in \Omega$,
$$\; t\mapsto \frac{g(x,t)}{t^{\frac{\gamma n}{2s}-1}} \;\text{is increasing for}\; t>0 \; \text{and}\; \ds\lim_{t\rightarrow 0^+} \frac{g(x,t)}{t^{\frac{\gamma n}{2 s}-1}}=0,$$ uniformly in $x\in \Om$. Also, for each $(x,t)\in \Om \times \mb R$ we have
\[(l+1)G(x,t)\leq tg(x,t).\]
\end{Remark}
%%%%%%%%%%%%%%%%%%%%%%%%%%%%%%%%%%%%%%%%%%%%%%%%%%%%%%%%
\begin{Example}
Let $g(x,t)= h(x,t)e^{|t|^{\frac{n}{n-s}}}$, where $h(x,t)=\left\{\begin{array}{lr}
0 \; \mbox{if} \; t\leq 0 \\
{t^{\al+ \left(\frac{\gamma n}{2s}-1\right)} \exp(dt^{\ba})}\; \mbox{if} \; t> 0.
\end{array}\right.$
for some $\al>0$, $0<d\leq \alpha_{n,s}$ and $1\leq \beta <\frac{n}{n-s}$. Then $g$ satisfies all the conditions from $(g1)-(g4)$.
\end{Example}
%%%%%%%%%%%%%%%%%%%%%%%%%%%%%%%%%%%%%%%%%%%%%%%%%%%%%%%%%%%%%%%%%%%%%%%%%%%%%%%%%%%%%%%%%%%%%%%%%%%%%%%%5
\begin{Definition}
We say that $u\in X_0$ is a weak solution of $(\mc{M})$ if, for all $\phi \in X_0$, it satisfies
\begin{align*}M(\|u\|^{\frac{n}{s}})\int_{\mb R^{2n}}\frac{|u(x)-u(y)|^{\frac{n}{s}-2}(u(x)-u(y))(\phi(x)-\phi(y))}{|x-y|^{2n}}dxdy =\int_\Om \left(\int_{\Om}\frac{G(y,u)}{|x-y|^{\mu}} dy \right) g(x,u)\phi~ dx.
\end{align*}
\end{Definition}

%%%%%%%%%%%%%%%%%%%%%%%%%%%%%%%%%%%%%%%%%%%%%%%%%
\noi Before stating our main Theorem, we recall a result of \cite{ruf} which will be used to find an upper bound for the Mountain Pass critical level.
Assume that $0\in \Omega$ and $B_1(0)\subset \Om$.  Then we consider the following Moser type functions which is given by equation $(5.2)$ of \cite{ruf}.
For each $x \in \mb R^n$ and $k \in \mb N$,
\begin{equation}\label{5.2}
\tilde{w_k}(x)=\left\{
    \begin{split}
        &|\log k|^{\frac{n-s}{n}}, \;\;\mbox{if}\; 0\leq |x|\leq \frac{1}{k},\\
        & \frac{|\log (|x|)|}{|\log(1/k)|^{s/n}},\;\;\mbox{if}\; \frac{1}{k} \leq |x|\leq 1,\\
        & 0,\; \;\mbox{if}\; |x|\geq 1,
    \end{split}
    \right.
\end{equation}
then supp$(\tilde{w}_k) \subset  B_1(0)\subset \Omega$ and $\tilde{w}_k|_{B_1(0)}\in W^{s,\frac{n}{s}}_{0}(B_1(0))$.\\
\noi Now by Proposition $5.1$ of \cite{ruf} we know that
\begin{equation}\label{mf}
    \lim_{k \to \infty}\|\tilde{w}_k\|^{\frac{n}{s}}= \lim_{k \to \infty} \int_{\mb R^n}\int_{\mb R^n}\frac{|\tilde{w}_k(x)-\tilde{w}_k(y)|^{\frac{n}{s}}}{|x-y|^{2n}}dx dy = \gamma_{n,s},
\end{equation}
where
\[\gamma_{n,s}:= \frac{2(n\mc W_n)^2 \Gamma(\frac{n}{s}+1)}{n!}\sum_{i=0}^{\infty} \frac{(n+i-1)!}{i! (n+2i)^{\frac{n}{s}}}.\]
\noi where $\mc W_n$ denotes the volume of $n$-dimensional unit {sphere}.
%%%%%%%%%%%%%%%%%%%%%%%%%%%%%%%%%%%%%%%%%%%%%%%%%%%%%%%%%%%%%%%%%%%%%%%%%%%%%%%%%%
\noi We also recall the following result of Lions known as higher integrability Lemma in case of fractional Laplacian, proved in \cite{perara}.
\begin{Lemma}\label{plc}
Let $\{v_k : \|v_k\|=1\}$ be a sequence in $W^{s,n/s}_{0}(\Om)$
converging weakly to a non-zero function $v$. Then for every $p$ such that
$p<\alpha_{n,s}(1-\|v\|^{\frac{n}{s}})^{\frac{-s}{n-s}}$,\[\sup_{k}\int_{\Om} \exp({p
|v_k|^{\frac{n}{n-s}}})< +\infty.\]
\end{Lemma}
Now we state our main result:
\begin{Theorem}\label{thm711}
Suppose $(M1)-(M3)$ and $(g1)-(g4)$ hold. Assume in addition that for {$\ba> \frac{2\al_{n,s}^{*}}{\al_{n,s}}$,}
\begin{align}\label{h-growth}
\lim_{t\ra +\infty} \frac{t g(x,t) G(x,t)}{\exp\left(\ba t^{\frac{n}{n-s}}\right)}=\infty \;\mbox{uniformly in}\; x\in \overline{\Om}.
\end{align}
 Then, problem $(\mc {M})$ admit a non negative non trivial solution.
\end{Theorem}
%{\color{green} We prove this Theorem via the Mountain Pass Lemma.}

%{\begin{Remark}
%In the special case, when $n=1$ and $s=1/2$, $\alpha_{1,1/2}= \alpha_{1,1/2}^*= 2\pi^2$. Then it is enough to assume $\beta \geq 2$ in \eqref{h-growth}.
%\end{Remark}}
\section{Proof of Main result}

We begin this section with the study of mountain pass structure and Palais-Smale sequences corresponding to the energy functional  $J:  X_{0}\rightarrow \mb R$ associated to the problem $(\mc M)$ {which} is defined as
\[ J(u)=\frac{s}{n}\hat M(\|u\|^{\frac{n}{s}})-\frac{1}{2}\int_\Om \left(\int_{\Om}\frac{G(y,u)}{|x-y|^{\mu}}dy\right)  G(x,u)~ dx.\]
 From the assumptions, $(g1)-(g4)$, we obtain that for any $\e>0$, $r\geq 1$, $1\leq \alpha <l+1$  there exists $C(\e)>0$ such that
\begin{align}\label{k1}
|G(x,t)| \le \e |t|^{\alpha} + C(\e) |t|^r \exp((1+\e)|t|^{\frac{n}{n-s}}),\;\; \text{for all}\; (x,t)\in \Omega \times \mb R.
\end{align}
Now by Proposition \ref{HLS}, for any $u \in X_0$ we obtain
\begin{align}\label{k2}
\int_{\Om}\left(\int_\Om \frac{G(y,u)}{|x-y|^{\mu}}dy\right)G(x,u)~dx  \leq C(n,\mu){\|G(\cdot,u)\|_{L^\frac{2n}{2n-\mu}(\Om)}^2}.
\end{align}
This implies that $J$ is well defined {using Theorem \ref{moser}}. Also one can easily see that $J$ is Fr$\acute{e}$chet differentiable and the critical points of $J$ are the weak solutions of $(\mc M)$.
%%%%%%%%%%%%%%%%%%%%%%%%%%%%%%%%%%%%%%%%%%%%%%%%%%%%%%%%%%%%%%%%%%%%%%%%%%%%%%%%%5
\begin{Lemma}\label{lem7.1}
 Assume that the conditions $(M1)$ and $(g1)-(g4)$ hold. Then $J$ satisfies the Mountain Pass geometry around $0$.
\end{Lemma}
\begin{proof}
From \eqref{k1}, \eqref{k2}, H\"{o}lder inequality and Sobolev embedding, we have
{\small\begin{align}\label{kc-MP1}
 &\int_{\Om}\left(\int_\Om \frac{G(y,u)}{|x-y|^{\mu}}dy\right)G(x,u)~dx \notag\\
 &\leq C(n,\mu)2^{2}\left(\e^{\frac{2n}{2n-\mu}} \int_\Om |u|^{\frac{2n\alpha}{2n-\mu}} + (C(\e))^{\frac{2n}{2n-\mu}} \int_\Om |u|^{\frac{2rn}{2n-\mu}}\exp\left(\frac{2n(1+\e)}{2n-\mu}|u|^{\frac{n}{n-s}} \right) \right)^{\frac{2n-\mu}{n}} \notag\\
 &\leq C \left({\e^{\frac{2n}{2n-\mu}} } \int_\Om |u|^{\frac{2n\alpha}{2n-\mu}} + C_1(\e) \|u\|^{\frac{2r n}{2n-\mu}} \left( \int_\Om\exp\left(\frac{4n(1+\e)\|u\|^{\frac{n}{n-s}}}{2n-\mu}\left(\frac{|u|}{\|u\|}\right)^{\frac{n}{n-s}}\right)\right)^{\frac12} \right)^{\frac{2n-\mu}{n}}.
% &\leq C_1 \left(\e \|u\|^{2\alpha} + C_2(\e) \|u\|^{2r} \left( \int_\Om\exp\left(\frac{4n(1+\e)\|u\|^{\frac{n}{n-s}}}{2n-\mu}\left(\frac{|u|}{\|u\|}\right)^{\frac{n}{n-s}}\right)\right)^{^{\frac{2n-\mu}{2n}}} \right) .
 \end{align}}
So if we choose $\e>0$ small enough  and $u$ such that $\displaystyle\frac{4n(1+\e)\|u\|^{\frac{n}{n-s}}}{2n-\mu} \leq \alpha_{n,s}$ then using the fractional Trudinger-Moser inequality \eqref{TM-ineq} in \eqref{kc-MP1}, we obtain
\begin{align*}
\int_{\Om}\left(\int_\Om \frac{G(y,u)}{|x-y|^{\mu}}dy\right)G(x,u)~dx &\leq C_2(\e) \left( \|u\|^{\frac{2n\alpha}{2n-\mu}}  +  \|u\|^{\frac{2rn}{2n-\mu}}
 \right)^{\frac{2n-\mu}{n}}\\
 & \leq C_3(\e) \left( \|u\|^{2\alpha}  +  \|u\|^{2r} \right).
 \end{align*}
 Using \eqref{cnd-M1} and above estimate, we have
\begin{align*}
J(u) &\geq \frac{s}{n}\hat M(1) \|u\|^{\frac{\gamma n}{s}}-   C_3(\e) \left( \|u\|^{2 \alpha}  +  \|u\|^{2r}
 \right),
\end{align*}
when $\|u\|\leq 1$. Choosing $\alpha>\frac {\gamma n}{2s}$, $r>\frac{\gamma n}{2s}$ and $\rho>0$ such that $\rho<\min\left\{1,\left(\frac{\alpha_{n,s}(2n-\mu)}{4n(1+\e)}\right)^{\frac{n-s}{n}}\right\}$ we obtain $J(u) \geq \sigma >0$ for all $u\in X_0$ with $\|u\|=\rho$ and for some $\sigma>0$ depending on $\rho$.\\

\noi The condition $(g4)$ implies that  there exist some positive constants $C_1$ and $C_2$ such that
 \begin{equation}\label{new2}G(x,t) \geq C_1t^{l+1}-C_2\;\text{ for all}\; (x,t) \in \Om \times [0,\infty).
 \end{equation}
 Let $\phi \in X_0$ such that $\phi\geq  0$ and $\|\phi\|=1$ then by \eqref{new2} we obtain
\begin{align*}
\int_\Om \left(\int_\Om \frac{G(y,t \phi)}{|x-y|^{\mu}}dy\right)G(x,t \phi)~dx &\geq \int_\Om \int_\Om \frac{(C_1(t \phi)^{l+1}(y)-C_2)(C_1(t \phi)^{l+1}(x)-C_2)}{|x-y|^{\mu}}~dxdy\\
  & = C_1^2 t^{2(l+1)} \int_\Om \int_\Om \frac{\phi^{l+1}(y)\phi^{l+1}(x)}{|x-y|^\mu}~dxdy \\
  & \quad -2C_1C_2t^{l+1}\int_\Om \int_\Om\frac{\phi^{l+1}(y)}{|x-y|^\mu}~dxdy + C_2^2 \int_\Om \int_\Om \frac{1}{|x-y|^{\mu}}~dxdy.
\end{align*}
This together with \eqref{cnd-M2}, we obtain
\begin{align*}
J(t \phi) & \leq \frac{s}{n}{M}(1)\|t \phi\|^{\frac{\gamma n}{s}}- \frac{1}{2}\int_{\Om}\left( \int_\Om\frac{G(y, t \phi)}{|x-y|^{\mu}} dy\right)G(x, t \phi)~dx\\
&\leq C_3+ C_4t^{\frac{\gamma n}{s}} - C_5t^{2(l+1)}+C_6t^{l+1},
\end{align*}
where $C_i's$ are positive constants for $i=3,4,5,6$. This implies that $J(t\phi) \to -\infty$ as $t \to \infty$, since $l+1>\frac{\ga n}{2s}$. Thus there exists a $v_0\in X_0$ with $\|v_0\|> \rho$ such that $J(v_0)<0$. Therefore, $J$ satisfies Mountain Pass geometry near $0$. \QED
\end{proof}

%%%%%%%%%%%%%%%%%%%%%%%%%%%%%%%%%%%%%%%%%%%%%%%%%%%%%%%%%%%%%%%%%%%%%%%%%%%%%%%%%%%%%%%%%%%%%%%%%%%%%%%%%%%%%%

\noi Let $\ds \Gamma=\{\gamma\in C([0,1],X_0):\gamma(0)=0,J(\gamma(1))<0\}$ and define the Mountain Pass critical level
$\ds c_*=\inf_{\gamma\in \Gamma}\max_{t\in[0,1]}J(\gamma(t))$. Then by Lemma \ref{lem7.1} and the Mountain pass theorem we know that there exists a Palais Smale sequence $\{u_k\}\subset X_0$ for $J$ at $c_*$ that is
\[J(u_k) \to c_* \; \text{and}\; J^\prime(u_k) \to 0\;\text{as}\; k \to \infty.\]
%%%%%%%%%%%%%%%%%%%%%%%%%%%%%%%%%%%%%%%%%%%%%%%%%%%%%%%%%%%%%%%%%%%%%%%%%%%%%%%%%%%%%%%%%%%%%
\begin{Lemma}\label{lem712}
Every Palais-Smale sequence of $J$ is bounded in $X_0$.
\end{Lemma}
\begin{proof}
Let $\{u_k\} \subset X_0$ denotes a $(PS)_c$ sequence of $J$ that is
\begin{equation*}
J(u_k) \to c \; \text{and}\; J^{\prime}(u_k) \to 0\;\text{as}\; k \to \infty
\end{equation*}
for some $c \in \mathbb{R}.$ This implies
\begin{align}\label{kc-PS-bdd1}
&\frac{s\hat M(\|u_k\|^{\frac{n}{s}})}{n} - \frac12 \int_\Om \left(\int_\Om \frac{G(y,u_k)}{|x-y|^{\mu}}dy \right)G(x,u_k)~dx \to c \; \text{as}\; k \to \infty,\notag\\
&\left| M(\|u_k\|^{\frac{n}{s}})\int_{\mb R^{n}}\int_{\mb R^n}\frac{|u_k(x)- u_k(y)|^{\frac{n}{s}-2}(u_k(x)- u_k(y))(\phi(x)-\phi(y))}{|x-y|^{2n}}{dxdy}\right.\notag\\
&\quad\quad\quad\quad\quad\quad\quad\quad\quad\quad\quad\quad\quad\left. -\int_\Om \left(\int_\Om \frac{G(y,u_k)}{|x-y|^{\mu}}dy \right)g(x,u_k)\phi ~dx \right|\leq \e_k\|\phi\|
\end{align}
where $\e_k \to 0$ as $k\to \infty$. In particular, taking $\phi=u_k$ we get
\begin{equation}\label{kc-PS-bdd2}
\left| M(\|u_k\|^{\frac{n}{s}}) \|u_k\|^{\frac{n}{s}}-\int_\Om \left(\int_\Om \frac{G(y,u_k)}{|x-y|^{\mu}}dy \right)g(x,u_k)u_k ~dx \right|\leq \e_k\|u_k\|.
\end{equation}
Now Remark \eqref{rem2} gives us that
\begin{equation}\label{kc-PS-bdd3}
(l+1) \int_\Om \left(\int_\Om \frac{G(y,u_k)}{|x-y|^{\mu}}dy \right)G(x,u_k)~dx \leq \int_\Om \left(\int_\Om \frac{G(y,u_k)}{|x-y|^{\mu}}dy \right)g(x,u_k)u_k ~dx.
\end{equation}
%We notice that because of \eqref{kc-1} and Remark \ref{rem2.2} we get
%\[\int_\Om \int_\Om \frac{G(y,u_k)}{|x-y|^{\mu}}~dxdy \leq \left(\int_\Om |G(y,u_k)|^{q^\prime}dy\right)^{\frac{1}{q^\prime}} \int_\Om \left(\int_\Om \frac{1}{|x-y|^{\mu q}}dy\right)^{\frac{1}{q}}~dx \leq C_0(\|u_k\|^{\beta_0+1} + \|u_k\|^p )  \]
%where $q\in [1,2n/\mu)$.}
Then using  \eqref{kc-PS-bdd1}, \eqref{kc-PS-bdd2} along with \eqref{kc-PS-bdd3} and \eqref{cnd-M}, we get
\begin{align}\label{kc-PS-bdd4}
& J(u_k)- \frac{1}{2(l+1)}\langle J^\prime(u_k),u_k\rangle
=\frac{s}{n}\hat M(\|u_k\|^{\frac{n}{s}})- \frac{1}{2 (l+1)}M(\|u_k\|^{\frac{n}{s}})\|u_k\|^{\frac{n}{s}}\notag  \\
& \quad \quad-\frac12 \left[ \int_\Om \left(\int_\Om \frac{G(y,u_k)}{|x-y|^{\mu}}dy \right)G(x,u_k)~dx - \frac{1}{(l+1)}\int_\Om \left(\int_\Om \frac{G(y,u_k)}{|x-y|^{\mu}}dy \right)g(x,u_k)u_k ~dx\right]\notag\\
&{\geq \frac{s\hat M(\|u_k\|^{\frac{n}{s}})}{n}- \frac{M(\|u_k\|^{\frac{n}{s}})\|u_k\|^{\frac{n}{s}}}{2(l+1)}}\notag \\
& {\geq \left( \frac{s}{n \gamma }- \frac{1}{2(l+1)}\right) M(\|u_k\|^{\frac{n}{s}})\|u_k\|^{\frac{n}{s}}}.
\end{align}
To prove the Lemma, we assume by contradiction that $\{\|u_k\|\}$ is an unbounded sequence. Then without loss of generality, we can assume that, up to a subsequence, $\|u_k\| \to \infty$ and $\|u_k\|\geq \alpha >0$ for some $\alpha$ and for all $k$.
This along with \eqref{kc-PS-bdd4} and $(M3)$ gives us
\begin{equation}\label{kc-PS-bdd6}
J(u_k)- \frac{1}{2(l+1)}\langle J^\prime(u_k),u_k\rangle\geq  \left( \frac{s}{n \gamma}- \frac{1}{2 (l+1)}\right) \kappa\|u_k\|^{\frac{n}{s}}
\end{equation}
where $\kappa$ depends on $\alpha$. Also from \eqref{kc-PS-bdd1} and \eqref{kc-PS-bdd2} it follows that
\begin{equation}\label{kc-PS-bdd5}
J(u_k)- \frac{1}{2(l+1)}\langle J^\prime(u_k),u_k\rangle \leq C \left( 1+ \e_k \frac{\|u_k\|}{2{(l+1)}}\right)
\end{equation}
for some constant $C>0$. Therefore from \eqref{kc-PS-bdd6} and \eqref{kc-PS-bdd5} we get that
\[ \left( \frac{s}{n \gamma }- \frac{1}{2 (l+1)}\right)\kappa \|u_k\|^{\frac{n}{s}}  \leq C \left( 1+ \e_k \frac{\|u_k\|}{2{(l+1)}}\right)\]
which gives a contradiction because $l+1>\frac{\gamma n}{2s}$ and $\frac{n}{s}>1$. This implies that $\{u_k\}$ must be bounded in $X_0$. \hfill{\QED}
\end{proof}
%%%%%%%%%%%%%%%%%%%%%%%%%%%%%%%%%%%%%%%%%%%%%%%%%%%%%%%%%%%%%%%%%%%%%%%%%%%%%%%
Assume that $0\in \Omega$ and $\rho>0$ be such that $B_\rho(0)\subset \Om$. Then for $x \in \mb R^n$, we define $w_k(x):= \tilde{w}_{k}\left(\frac{x}{\rho}\right)$, where $\tilde w_k$ is same as \eqref{5.2} then supp$(w_k) \in B_\rho(0)\subset \Omega$. We note that  $w_k\in W_{0}^{s,\frac{n}{s}}(\mb R^n)$ and by \eqref{mf}, we have
\begin{equation}\label{MF-limit}
    \lim_{k \to \infty}\|w_k\|^{\frac{n}{s}}= \lim_{k \to \infty} \int_{\mb R^{n}}\int_{\mb R^n}\frac{|\tilde w_k(x)-\tilde w_k(y)|^{\frac{n}{s}}}{|x-y|^{2n}}dx dy = \gamma_{n,s}.
\end{equation}
Next, we use $w_k$'s efficiently  to obtain the following bound on $c_*$.
%%%%%%%%%%%%%%%%%%%%%%%%%%%%%%%%%%%%%%%%%%%%%%%%%%%%%
\begin{Lemma}\label{lem7.2}
It holds that
\[\ds 0<c_*<\frac{s}{n}\hat M\left(\left(\frac{2n-\mu}{2n}\alpha_{n,s}\right)^{\frac{n-s}{s}}\right).\]
\end{Lemma}
\begin{proof} Using Lemma \ref{lem7.1}, we deduce that $c_*>0$ and $J(t\phi) \ra -\infty$ as $t\ra \infty$ if $0\leq \phi \in X_0\setminus\{0\}$ with $\|\phi\|=1$. Also by definition of $c_*$, we have  $c_* \leq \max\limits_{t\in [0,1]}J(t\phi)$ for each non negative $\phi \in X_0\setminus \{0\}$ with $J(\phi)<0$ {which} assures that it is enough to prove that there exists a non negative $w \in X_0\setminus \{0\}$ such that
\[\max_{t\in[0,\infty)} J(tw) < \frac{s}{n}\hat M\left(\left( \frac{2n-\mu}{2n}\alpha_{n,s}\right)^{\frac{n-s}{s}}\right).\]
To prove this, we consider the sequence of non negative functions $\{w_k\}$(defined before this Lemma) and claim that there exists a $k \in \mb N$ such that
\[\max_{t\in[0,\infty)} J(tw_k) < \frac{s}{n}\hat M\left(\left( \frac{2n-\mu}{2n}\alpha_{n,s}\right)^{\frac{n-s}{s}}\right).\]
Suppose this is not true, then for all $k \in \mb N$ there exists a $t_k>0$ such that
\begin{equation}\label{kc-PScond0}
\begin{split}
&\max_{t\in[0,\infty)} J(tw_k) = J(t_kw_k) \geq \frac{s}{n} \hat M\left(\left( \frac{2n-\mu}{2n}\alpha_{n,s}\right)^{\frac{n-s}{s}}\right)\\
& \text{and}\;  \frac{d}{dt}(J(tw_k))|_{t=t_k}=0.
\end{split}
\end{equation}
From the proof of {Lemma \ref{lem7.1}}, $J(t w_k)\to -\infty$ as $t\to \infty$ for each $k$. Then we infer that $\{t_k\}$ must be a bounded sequence in $\mb R$ which implies that there exists a $t_0$ such that, up to a subsequence which we still denote by $\{t_k\}$, $t_k \to t_0$ as $k \to \infty$. From \eqref{kc-PScond0} and definition of $J(t_kw_k)$ we obtain
\begin{equation}\label{kc-PScond1}
\frac{s}{n} \hat M\left(\left( \frac{2n-\mu}{2n}\alpha_{n,s}\right)^{\frac{n-s}{s}}\right) < \frac{s}{n}\hat M(\|t_kw_k\|^{\frac{n}{s}}).
\end{equation}
Since $\hat M$ is monotone increasing, from \eqref{kc-PScond1} we get that
\begin{equation}\label{kc-PScond2}
\|t_kw_k\|^{\frac{n}{s}} \geq \left( \frac{2n-\mu}{2n}\alpha_{n,s}\right)^{\frac{n-s}{s}}.
\end{equation}
From \eqref{kc-PScond2} and since \eqref{MF-limit} holds, we infer that
\begin{equation}\label{kc-PScond3}
t_k(\log k)^{\frac{n-s}{n}} \to \infty \;\text{as}\; k \to \infty.
\end{equation}
Furthermore from \eqref{kc-PScond0},  we have
\begin{equation}\label{kc-PS-cond3}
\begin{split}
M(\|t_kw_k\|^{\frac{n}{s}})\|t_kw_k\|^{\frac{n}{s}} &= \int_\Om \left(\int_\Om \frac{G(y,t_kw_k)}{|x-y|^{\mu}}dy\right)g(x,t_kw_k)t_kw_k ~dx\\
& \geq \int_{B_{\rho/k}}g(x,t_kw_k)t_kw_k \int_{B_{\rho/k}}\frac{G(y,t_kw_k)}{|x-y|^\mu}~dy~dx.
%& \geq (\beta-\e)\exp\left( a \left( \frac{t_k(\log k)^{\frac{n-1}{n}}}{\omega_{n-1}^{\frac{1}{n}}}\right)^{\frac{n}{n-1}}\right)\int_{B_{\rho/k}}\int_{B_{\rho/k}} \frac{~dxdy}{|x-y|^\mu}
\end{split}
\end{equation}
 In addition, as in equation $(2.11)$ p. 1943 in \cite{YangJDE}, it is easy to get that
\[\int_{B_{\rho/k}}\int_{B_{\rho/k}} \frac{~dxdy}{|x-y|^\mu} \geq C_{\mu, n} \left(\frac{\rho}{k}\right)^{2n-\mu},\]
where $C_{\mu, n}$ is a positive constant depending on $\mu$ and $n$. From \eqref{h-growth}, {it is easy to deduce that for $\ba>\frac{2 \alpha_{n,s}^*}{\al_{n,s}}$ and for each $d>0$ there exists a $r_d\in \mb N$ }such that
\[rg(x,r)G(x,r) \geq d \exp\left( \beta|r|^{\frac{n}{n-s}}\right)\; \text{whenever}\; r \geq r_d.\]
Since \eqref{kc-PScond3} holds, we can choose a $N_d \in \mb N$ such that
\[t_k(\log k)^{\frac{n-s}{n}} \geq r_d\; \text{for all}\; k\geq N_d.\]
 Using these estimates in \eqref{kc-PS-cond3} and from \eqref{kc-PScond2}, for $d$ large enough
  we get that
\begin{align}\label{PS-con}
M(\|t_kw_k\|^{\frac{n}{s}})\|t_kw_k\|^{\frac{n}{s}}
&  \geq d \exp \left(\beta t_k^{\frac{n}{n-s}}|\log k|\right)C_{\mu,n}\left(\frac{\rho}{k}\right)^{2n-\mu}\nonumber\\
& = d C_{\mu,n}\rho^{2n-\mu} \exp\left(\left(\beta t_k^{\frac{n}{n-s}}-(2n-\mu)\right)\log k\right)\\
& \geq d C_{\mu,n}\rho^{2n-\mu} \exp \left( \log k \left( \frac{(2n-\mu)\beta \alpha_{n,s}}{2n\|w_k\|^{\frac{n}{n-s}}} - (2n-\mu)\right)\right)\nonumber
\end{align}
%which implies
%\begin{equation}
%\frac{m(t_k^n)t_k^n}{k^{{a t_k^{\frac{n}{n-1}}}{\omega_{n-1}^{-\frac{1}{n-1}}}-(4-\mu)}}\geq (\beta -\e)C_\mu\rho^{4-\mu}.
%\end{equation}
Since $\beta > \frac{2\alpha_{n,s}^*}{\alpha_{n,s}}=\frac{2n \gamma_{n,s}^{\frac{s}{n-s}}}{\alpha_{n,s}}$ and \eqref{MF-limit} hold, the R.H.S. of \eqref{PS-con} tends to $+\infty$ as $k \to \infty$. Whereas from continuity of $M$ it follows that
\[\lim_{k \to \infty} M\left(\|t_kw_k\|^{\frac{n}{s}}\right)\|t_kw_k\|^{\frac{n}{s}} = M\left(t_0^{\frac{n}{s}}\gamma_{n,s}\right)(t_0^{\frac{n}{s}}\gamma_{n,s}),\]
which is a contradiction. This establishes our claim and {we} conclude the proof of Lemma.
\QED
\end{proof}
%%%%%%%%%%%%%%%%%%%%%%%%%%%%%%%%%

\noi In order to prove that a Palais-Smale sequence converges to a weak solution of problem ($\mc M$), we need  the following  convergence Lemma. The idea of proof is borrowed from Lemma 2.4 in  \cite{YangJDE}.
\begin{Lemma}\label{wk-sol}
If $\{u_k\}$ is a Palais Smale sequence for $J$ at $c$ then there exists a $u \in X_0$ such that, up to a subsequence.
\begin{equation}\label{wk-sol1}
\left(\int_\Om \frac{G(y,u_k)}{|x-y|^{\mu}}~dy\right)G(x,u_k) \to \left(\int_\Om \frac{G(y,u)}{|x-y|^{\mu}}~dy\right)G(x,u) \; \text{in}\; L^1(\Om)
\end{equation}
\end{Lemma}
\begin{proof}
From Lemma \ref{lem712}, we know that the sequence $\{u_k\}$ must be bounded in $X_0$. Consequently, up to a subsequence, there exists a $u\in X_0$ such that $u_k \rightharpoonup u$ weakly in $X_0$ and strongly in $L^q(\Om)$ for any $q \in [1,\infty)$ as $k \to \infty$. Also, still up to a subsequence, we can assume  that $u_k(x) \to u(x)$ pointwise a.e. for $ x \in \Om$.

From \eqref{kc-PS-bdd1}, \eqref{kc-PS-bdd2} and \eqref{kc-PS-bdd3} we get that there exists a constant $C>0$ such that
\begin{equation}\label{wk-sol10}
\begin{split}
\int_\Om \left(\int_\Om\frac{G(y,u_k)}{|x-y|^\mu}dy\right)G(x,u_k)~dx &\leq C,\\
\int_\Om \left(\int_\Om\frac{G(y,u_k)}{|x-y|^\mu}dy\right)g(x,u_k)u_k~dx & \leq C.
\end{split}
\end{equation}
Now, it is well known that if $f\in L^1(\Om)$ then for any $\e>0$ there exists a $\delta>0$ such that
\[\left| \int_{U} f(x)~dx\right| <\e,\]
for any measurable set $U\subset \Om$ with $|U|\leq \delta$.
%Now Since $\left(\int_\Om\frac{G(y,u_k)}{|x-y|^\mu}dy\right)G(x,u_k)~dx \in L^{1}(\Om)$. It follows that for every $\e>0$, there exists $\de>0$ such that
%\[\int_{U} \left(\int_\Om\frac{G(y,u_k)}{|x-y|^\mu}dy\right)G(x,u_k)~dx \leq \e\;\forall\; |U|\leq \de\]
%for all measurable subset $U$ of $\Om$.
Also $f\in L^{1}(\Om)$ implies that for any fixed $\delta>0$ there exists $M>0$ such that
\[|\{x\in \Om : |f(x)|\geq M\}|\leq \delta.\]
Now using \eqref{wk-sol10}, we have
$$\left(\displaystyle\int_\Om\frac{G(y,u_k)}{|x-y|^\mu}dy\right)G(\cdot,u_k) \in L^{1}(\Om)$$
and also by \eqref{k2}
$$\left(\displaystyle\int_\Om\frac{G(y,u)}{|x-y|^\mu}dy\right)G(\cdot,u) \in L^{1}(\Om).$$
Now we fix $\de>0$ and  choose $M> \max\left\{\left(\frac{C T_0}{\de}\right)^{\frac{1}{\ga_0+1}}, t_0\right\}$. Then we use $(g3)$  to obtain
\begin{align*}
 \int_{\Om \cap\{u_k \geq M\}} &\left(\int_{\Omega} \frac{G(y,u_k)}{|x-y|^{\mu}} dy \right)G(x,u_k)  ~dx \leq T_0 \int_{\Om \cap\{u_k \geq M\}} \left(\int_{\Omega} \frac{G(y,u_k)}{|x-y|^{\mu}}dy\right) \frac{g(x,u_k)}{u_{k}^{\gamma_0}}  ~dx\\
 &\leq \frac{T_0}{ M^{\gamma_0+1}}\int_{\Om \cap\{u_k \geq M\}} \left(\int_{\Omega} \frac{G(y,u_k)}{|x-y|^{\mu}} dy\right) g(x,u_k)u_{k}  ~dx<\de.
\end{align*}
Next we consider
\begin{align*}
&\left| \int_{\Omega} \left(\int_{\Omega}\frac{G(y,u_k)}{|x-y|^{\mu}} dy \right) G(x,u_k) ~dx-  \int_{\Omega} \left(\int_{\Omega}\frac{G(y,u )}{|x-y|^{\mu}} dy \right) G(x,u) ~dx\right|\\
&\leq 2\de+ \left|\int_{\Om \cap \{u_k \leq M\}} \left(\int_{\Omega}\frac{G(y,u_k)}{|x-y|^{\mu}} dy \right) G(x,u_k) ~dx-  \int_{\Omega\cap \{u \leq M\}} \left(\int_{\Omega}\frac{G(y,u )}{|x-y|^{\mu}} dy \right) G(x,u) ~dx\right|
\end{align*}
To prove the result, it is enough to establish that as $k \to \infty$
\begin{align}
\int_{\Om \cap \{u_k \leq M\}} \left(\int_{\Omega}\frac{G(y,u_k)}{|x-y|^{\mu}} dy \right) G(x,u_k)~dx \to \int_{\Om \cap \{u \leq M\}} \left(\int_{\Omega}\frac{G(y,u )}{|x-y|^{\mu}} dy \right) G(x,u) ~dx.
\end{align}
Since $\left(\displaystyle\int_\Om\frac{G(y,u)}{|x-y|^\mu}dy\right)G(\cdot,u) \in L^{1}(\Om)$, so by Fubini's theorem we get
\begin{align*}
%\lim_{K \to \infty} \int_{|u_k|\leq M}\left(\int_{|u_k|\geq K}\frac{G(x,u_k)}{|x-y|^{\mu}}dy\right)G(x,u)~dx =
&\lim_{K \to \infty} \int_{\Om \cap\{u\leq M\}}\left(\int_{\Om\cap\{u\geq K\}}\frac{G(y,u)}{|x-y|^{\mu}}dy\right)G(x,u)~dx\\
&= \lim_{K \to \infty} \int_{\Om \cap\{u\geq K\}}\left(\int_{\Om \cap\{u\leq M\}}\frac{G(y,u)}{|x-y|^{\mu}}dy\right)G(x,u)~dx=0.
\end{align*}
Thus we can fix a  $K> \max\left\{\left(\frac{C T_0}{\de}\right)^{\frac{1}{\ga_0+1}}, t_0\right\}$  such that
\[\int_{\Om \cap\{u \leq M\}} \left(\int_{\Om \cap\{u\geq K\} }\frac{G(y,u )}{|x-y|^{\mu}} dy \right) G(x,u) ~dx \leq \delta.\]
From $(g3)$, we get
\begin{align*}
&\int_{\Om \cap\{u_k \leq M\}} \left(\int_{\Om \cap\{u_k\geq K\}}\frac{G(y,u_k )}{|x-y|^{\mu}} dy \right) G(x,u_k) ~dx\\
& \leq\frac{1}{K^{\gamma_0+1}}\int_{\Om \cap \{u_k \leq M\}} \left(\int_{\Om \cap\{u_k\geq K\}}\frac{u_k^{\gamma_0+1}(y)G(y,u_k )}{|x-y|^{\mu}} dy \right) G(x,u_k) ~dx\\
&\leq \frac{T_0}{K^{\ga_{0}+1}} \int_{\Om \cap\{u_k \leq M\}} \left(\int_{\Om \cap\{u_k\geq K\} }\frac{ u_k(y) g(y,u_k )}{|x-y|^{\mu}} dy \right) G(x,u_k) ~dx\\
&\leq \frac{T_0}{K^{\ga_{0}+1}} \int_{\Om} \left(\int_{\Om }\frac{G(y,u_k )}{|x-y|^{\mu}} dy \right) g(x,u_k) u_k ~dx\leq \de.
\end{align*}
Thus we have proved that
\begin{align*}
    &\left|\int_{\Om \cap\{u \leq M\}} \left(\int_{\Om \cap\{u\geq K\} }\frac{G(y,u )}{|x-y|^{\mu}} dy \right) G(x,u) ~dx\right.\\
    &\quad \quad \left.- \int_{\Om \cap\{u_k \leq M\}} \left(\int_{\Om \cap\{u_k\geq K\} }\frac{G(y,u_k)}{|x-y|^{\mu}} dy \right) G(x,u_k) ~dx\right|\leq 2\de
\end{align*}
Finally, { to complete the proof of Lemma, we need to verify that as $k\ra \infty$}
\begin{equation}\label{choq-new}
\begin{split}
&\left|\int_{\Om\cap\{u_k \leq M\}} \left(\int_{\Om \cap\{u_k\leq K\}}\frac{G(y,u_k)}{|x-y|^{\mu}} dy \right) G(x,u_k) ~dx-\right.\\
&\quad \left.\int_{\Om \cap\{u \leq M\}} \left(\int_{\Om \cap\{u\leq K\}}\frac{G(y,u )}{|x-y|^{\mu}} dy \right) G(x,u) ~dx \right|\ra 0
\end{split}
\end{equation}
for fixed positive $K$ and $M$. It is easy to see that
\begin{align*}
\left(\int_{\Om \cap\{u_k\leq K\} }\frac{G(y,u_k)}{|x-y|^{\mu}} dy \right) G(x,u_k)\chi_{ \Om \cap \{u_k\leq M\}}  &\ra \left(\int_{\Om \cap \{u\leq K\} }\frac{G(y,u)}{|x-y|^{\mu}} dy \right) G(x,u)\chi_{\Om \cap \{u \leq M\}}
\end{align*}
pointwise a.e. as $k \to \infty$. Now choose $r=\alpha$ in \eqref{k1}, which gives us that there exist a constant $C_{M,K}>0$ depending on $M$ and $K$ such that
\begin{align*}
 &\int_{\Om \cap \{u_k\leq M\}}\left( \int_{\Om \cap \{u_k\leq K\} }\frac{G(y,u_k)}{|x-y|^{\mu}} dy \right)  G(x,u_k)dx  \\
   &\leq  C_{M,K}\int_{\Om \cap \{u_k\leq M\}}\left( \int_{\{u_k\leq K\} }\frac{|u_k(y)|^{r}}{|x-y|^{\mu}} dy \right)  |u_k(x)|^{r} dx \\
   & \leq C_{M,K} \int_\Om\int_{\Om }\left(\frac{|u_k(y)|^{r}}{|x-y|^{\mu}}~dy  \right) |u_k(x)|^{r} ~dx\\
   & \leq { {C_{M,K}C(n,\mu)\|u_k\|_{L^{\frac{2nr}{2n-\mu}}(\Om)}^{2r} \to C_{M,K}C(n,\mu)\|u\|_{L^{\frac{2nr}{2n-\mu}}(\Om)}^{2r}}} \; \text{as}\; k \to \infty,
    %&\ra \left(\int_{\Om } \frac{|u(y)|^{p+1}}{|x-y|^{\mu}} dy \right) u^{p+1}(x)\chi_{\Om} ~dx
\end{align*}
where we used the Hardy-Littlewood-Sobolev inequality in the last inequality and then used the fact that $u_k \to u$ strongly in $L^q(\Om)$ for each $q \in [1,\infty)$. This implies that, using Theorem $4.9$ of \cite{Brezis-book}, there exists a constant $h \in L^1(\Om)$ such that, up to a subsequence, for each $k$
\[\left|\left( \int_{\Om \cap\{u_k\leq K\} }\frac{G(y,u_k)}{|x-y|^{\mu}} dy \right)  G(x,u_k)\chi_{ \Om \cap \{u_k\leq M\}} \right| \leq |h(x)|\]
This helps us to employ the Lebesgue dominated convergence theorem and conclude \eqref{choq-new}.  \hfill{\QED}
\end{proof}
%%%%%%%%%%%%%%%%%%%%%%%%%%%%%%%%%%%%%%%%%%%%%5

\begin{Lemma}\label{PS-ws}
Let $\{u_k\}\subset X_{0}$ be a Palais Smale sequence of $J$. Then there exists a $u \in X_0$ such that, up to a subsequence, for all $\phi\in X_0 $
\begin{equation}\label{PS-wk0}
\int_{\Om}\left(\int_\Om \frac{G(y,u_k)}{|x-y|^{\mu}}dy\right)g(x,u_k)\phi~ dx  \to \int_{\Om}\left(\int_\Om \frac{G(y,u)}{|x-y|^{\mu}}dy\right)g(x,u)\phi~dx \; \text{as}\; k \to \infty\;.
\end{equation}
\end{Lemma}
\begin{proof}
{As we argued} in previous Lemma, we have that {there exists a $u\in X_0$ such that,} up to a subsequence, $u_k \rightharpoonup u$ weakly in $X_0$, $u_k \to u$ pointwise a.e. in $\mb R^n$,  $\|u_k\| \to \tau$ as $k \to \infty$, for some $\tau \geq 0$ and $u_k \to u$ strongly in $L^q(\Om)$, $q \in [1,\infty)$ as $k \to \infty$.

Let $\Om' \subset\subset \Om$ and $\varphi \in C_c^\infty(\Om)$ such that $0\leq \varphi \leq 1$ and $\varphi \equiv 1$ in $\Om' $. Then by taking $\varphi$ as a test function in \eqref{kc-PS-bdd1}, we obtain the following estimate
\begin{equation*}
\begin{split}
&\int_{\Om^{'}}\left( \int_\Om \frac{G(y,u_k)}{ |x-y|^\mu}dy\right)g(x,u_k)~dx \leq \int_\Om \left( \int_\Om \frac{G(y,u_k)}{ |x-y|^\mu}dy\right)g(x,u_k)\varphi~dx\\
&\leq  \e_k \left\|\varphi\right\| + M(\|u_k\|^{\frac {n}{s}}) \int_{\mb R^{2n}} \frac{|u_{k}(x)-u_k(y)|^{\frac{n}{s}-2}(u_k(x)-u_k(y))(\varphi(x)-\varphi(y))}{|x-y|^{2n}}~dxdy\\
& \leq \e_k \|\varphi\|+ C \|u_k\| \|\varphi\| {\leq C},
\end{split}
\end{equation*}
since $\|u_k\|\leq C_0$ for all $k$.  This implies that the sequence $\{\mu_k\}:=\left\{\left( \int_\Om\frac{G(y,u_k)}{ |x-y|^\mu}dy\right)g(x,u_k)\right\}$ is bounded in $L^1_{\text{loc}}(\Om)$ which implies that up to a subsequence, $\mu_k \to \mu$ in the ${weak}^*$-topology as $k \to \infty$, where $\mu$ denotes a Radon measure. So for any $\phi \in C_c^\infty(\Om)$ we get
\[\lim_{k \to \infty}\int_\Om \left( \int_\Om\frac{G(y,u_k)}{|x-y|^\mu}dy\right)g(x,u_k)\phi ~dx = \int_\Om \phi ~d\mu,\; \forall \;\phi \in C_c^\infty(\Om). \]
Since $u_k$ satisfies \eqref{kc-PS-bdd1},  for any measurable set $E \subset \Om$, {taking $\phi \in C_c^\infty(\Om)$ such that supp$\phi \subset E$, we get that
\begin{align*}
\mu(E)&=\int_E \phi~ d\mu= \lim_{k \to \infty} \int_E\int_\Om \left( \frac{G(y,u_k)}{|x-y|^\mu}dy\right)g(x,u_k)\phi(x) ~dx \\
&= \lim_{k \to \infty} \int_\Omega\int_\Om \left( \frac{G(y,u_k)}{|x-y|^\mu}dy\right)g(x,u_k)\phi(x) ~dx\\
&= \lim_{k \to \infty} M(\|u_k\|^{\frac{n}{s}})\int_{\mb R^{n}}\int_{\mb R^{n}} \frac{|u_{k}(x)-u_k(y)|^{\frac{n}{s}-2}(u_k(x)-u_k(y))(\phi(x)-\phi(y))}{|x-y|^{2n}}~dx dy\\
& {= M(\tau^{\frac{n}{s}}) \int_{\mb R^{n}}\int_{\mb R^{n}} \frac{|u(x)-u(y)|^{\frac{n}{s}-2}(u(x)-u(y))(\phi(x)-\phi(y))}{|x-y|^{2n}}~dx dy,}
\end{align*}
where we used the continuity of $M$} and weak convergence of $u_k$ to $u$ in $X_0$. This implies that $\mu$ is absolutely continuous with respect to the Lebesgue measure. Thus, Radon-Nikodym theorem establishes that there exists a function $h \in L^1_{\text{loc}}(\Om)$ such that for any $\phi\in C^\infty_c(\Omega)$, $\int_\Om \phi~ d\mu = \int_\Om \phi h~dx$.
Therefore for any $\phi\in C^\infty_c(\Omega)$ we get
\[\lim_{k \to \infty}\int_\Om\left( \int_\Om \frac{G(y,u_k)}{|x-y|^\mu}dy\right)g(x,u_k)\phi~ ~dx = \int_\Om \phi h~dx = \int_\Om  \left( \int_\Om \frac{G(y,u)}{|x-y|^\mu}dy\right)g(x,u)\phi~ ~dx \]
and the above holds for any $\phi \in X_0$ using the density argument. This completes the proof.\hfill{\QED}
\end{proof}

%%%%%%%%%%%%%%%%%%%%%%%%%%%%%%%%%%%%%%%%%%%%%%%%%%%%%%%%%%%%%%%%%%%%%%%%%%%%%5%

\noi Now we define the Nehari manifold associated to the functional $J$, as
\[ \ds \mc{N}:=\{0\not\equiv u\in X_0:\langle J^{\prime}(u),u \rangle=0\}\]
and let $\ds b:=\inf_{u\in \mc{N}} J(u)$. Then we need the following Lemma to compare  $c_*$ and $b$.
\begin{Lemma}\label{lem7.3}
If condition $(g4)$ holds, then for each $x\in\Omega$, $ tg(x,t)-\frac{\gamma n}{2 s} G(x,t)$ is increasing for $t\ge0$. In particular $tg(x,t)-\frac{\gamma n}{2s} G(x,t)\geq 0$ for all $(x,t)\in \Omega\times [0,\infty)$ which implies $\frac{G(x,t)}{t^{\frac{\gamma n}{2 s}}}$ is non-decreasing for $t>0$.
\end{Lemma}
\proof Suppose $0<t<r$. Then for each $x\in\Om$, we obtain
 \begin{align*}
tg(x,t)-\frac{\gamma n}{2 s}G(x,t)&=\frac{g(x,t)}{t^{l}}t^{l+1}-\frac{\gamma n}{2s}G(x,r)+\frac{\gamma n}{2s}\int_{t}^{r} g(x,\tau)d\tau\\
&< \frac{g(x,t)}{t^{l}}t^{l+1} - \frac{\gamma n}{2s} G(x,r)+\frac{\gamma n }{2s(l+1)}\frac{g(x,r)}{r^{l}}(r^{l+1}-t^{l+1})\\
&\leq rg(x,r)- \frac{\gamma n}{2s} G(x,r),
\end{align*}
using $(g4)$. This completes the proof.\QED
%%%%%%%%%%%%%%%%%%%%%%%%%%%%%%%%%%%%%%%%%%%%%%%%%%%%%%%%%%%%%%%%%%%%%%%%%%%%%%%%%%%%%%%
\begin{Lemma}\label{3.7}
Under the assumptions $(M2)$ and $(g4)$, it holds $c_*\leq b $.
\end{Lemma}
\begin{proof} Let $u\in \mc{N}$ { be non negative}  and we define $h:{(0,\infty)}\rightarrow \mb R$ by $ h(t)=J(tu)$. Then for all $t>0$
\[h^{\prime}(t)=\langle J^{\prime}(tu),u\rangle= {M}(t^{\frac{n}{s}}\|u\|^{\frac{n}{s}})t^{\frac{n}{s}-1}\|u\|^{\frac{n}{s}}-\int_\Om \left(\int_\Om \frac{G(y,tu)}{|x-y|^{\mu}}dy \right)g(x,tu)u~dx .\]
Since $\langle J^{\prime}(u),u\rangle=0$ and $t\mapsto \frac{g(x,t)}{t^{\frac{\gamma n}{2s}-1}}$ is increasing for $t>0$, we have
{\small
\begin{align*}
h^{\prime}(t)=&\|u\|^{\frac{\gamma n}{s}} t^{\frac{\gamma n}{s}-1}\left(\frac{M(t^{\frac{n}{s}}\|u\|^{\frac{n}{s}})}{t^{(\gamma -1)\frac{n}{s}}\|u\|^{(\gamma-1)\frac{n}{s}}}-\frac{M(\|u\|^{\frac{n}{s}})}{\|u\|^{(\gamma -1)\frac{n}{s}}}\right)\\
&\quad +t^{\frac{\gamma n}{s}-1}\int_{\Om}\left(\int_\Om \frac{\frac{G(y,u)g(x,u)}{u^{\frac{\gamma n}{2s}-1}(x)}}{|x-y|^{\mu}}dy-\int_\Om \frac{\frac{G(y,tu)g(x,tu)}{(tu)^{\frac{\gamma n}{2s}-1}(x)t^{\frac {\gamma n}{2s}}}}{|x-y|^{\mu}}dy \right)u^{\frac{\gamma n}{2s}}(x)dx\\
&\geq \|u\|^{\frac{\gamma n}{s}} t^{\frac{\gamma n}{s}-1}\left(\frac{M(t^{\frac{n}{s}}\|u\|^{\frac{n}{s}})}{t^{\frac{(\gamma -1)n}{s}}\|u\|^{\frac{(\gamma-1)n}{s}}}-\frac{M(\|u\|^{\frac{n}{s}})}{\|u\|^{\frac{(\gamma -1)n}{s}}}\right)\\
&\quad +t^{\frac{\gamma n}{s}-1}\int_{\Om}\left(\int_\Om \left( G(y,u)- \frac{G(y,tu)}{t^{\frac{\gamma n}{2s}}}\right)\frac{1}{|x-y|^{\mu}}dy \right) \frac{g(x,tu)}{(tu)^{\frac{\gamma n}{2s}-1}(x)}u^{\frac{\gamma n}{2s}}(x)dx.\end{align*} }
when $0<t<1$. So using Lemma \ref{lem7.3} and $(M2)$ we have $h^{\prime}(1)=0$, $h^{\prime}(t)\geq 0$ for $0<t<1$ and $h^{\prime}(t)<0$ for $t>1$. Hence $J(u)=\ds \max_{t\geq0}J(tu)$. Now define $f:[0,1]\rightarrow X_0$ as  $f(t)=(t_0 u)t$, where $t_0>1$ is such that $J(t_0 u)<0$. Then we have $f\in \Gamma$ and therefore
\[ c_*\leq\max_{t\in[0,1]}J(f(t))\leq \max_{t\geq 0}J(tu)=J(u)\leq \inf_{u \in \mc N}J(u)= b.\]
Hence the proof is complete.\QED
  \end{proof}

\begin{Definition} A solution $u_0$ of $(\mc M)$ is a ground state if $u_0$ is a weak solution of $(\mc M)$ and satisfies $J(u_0)=\ds \inf _{u\in \mc N} J(u)$.
\end{Definition}
Since $c_*\leq b$ in order to obtain a ground state solution $u_0$ for $(\mc M)$, it is enough to show that there exists a weak solution of $(\mc M)$ such that $J(u_0)=c_*$.

\begin{Lemma}\label{pos-sol}
Any nontrivial solution of problem $(\mc M)$ is nonnegative.
\end{Lemma}
\begin{proof}
Let $u\in X_0\setminus \{0\}$ be a critical point of functional $J$. Clearly $u^{-}={\max\{-u,0\}}\in X_0$. Then $\langle J^{\prime}(u),u^{-}\rangle=0$, i.e.
\begin{align*}
&{M}(\|u\|^{\frac{n}{s}})\int_{\mb R^{2n}}\frac{|u(x)-u(y)|^{\frac{n}{s}-2}(u(x)-u(y))(u^-(x)-u^-(y))}{|x-y|^{2n}}dxdy\\
&\quad \quad=  \int_{\Om}\left(\int_{\Om}\frac{G(y,u)}{|x-y|^{\mu}}dy \right)  g(x,u)u^- dx.
\end{align*}
For a.e. $x,y\in \mb R^n$, using $|u^-(x)-u^-(y)|\leq |u(x)-u(y)|$, we have
\begin{align*}
&|u(x)-u(y)|^{\frac{n}{s}-2} (u(x)-u(y))(u^-(x)-u^-(y))\\
&=  -|u(x)-u(y)|^{\frac{n}{s}-2} (u^{+}(x) u^{-}(y) + u^{-}(x)u^{+}(y) +|u^-(x)-u^-(y)|^{2})\\
&\leq - |u^-(x)-u^-(y)|^{\frac{n}{s}}
\end{align*}
and $g(x,u) u^-=0$ a.e. $x\in \Om$ by assumption. Hence,
\begin{align*}
0\leq - {M}(\|u\|^{\frac{n}{s}})\|u^{-}\|^{\frac{n}{s}}\leq 0.
\end{align*}
So, $u^- \equiv 0$ since $\|u\|>0$ and $(M3)$ holds. Hence $u\geq 0$ a.e. in $\Om$.\QED
\end{proof}

%%%%%%%%%%%%%%%%%%%%%%%%%%%%%%%%%%%%%%%%%%%%%%%%%%%%%%%%%%%%%%%%%%%%%%%%5
\noi {\textbf{Proof of Theorem \ref{thm711}:} Since $J$ satisfies the Mountain Pass geometry (refer Lemma \ref{lem7.1}), by Mountain Pass Lemma we know that there exists a Palais Smale $\{u_k\}$ sequence for $J$ at $c_*$. Then by Lemma \ref{lem712}, $\{u_k\}$ must be bounded in $X_0$ so that, up to a subsequence, $u_k \rightharpoonup u_0$ weakly in $X_0$, strongly in $L^q(\Om)$ for $ q\in [1, \infty)$, pointwise a.e. in $\Om$, {for some $u_0\in X_0$} and $\|u_k\| \to \rho_0\geq 0$ as $k \to \infty$.}\\
\noi \textbf{Claim 1:} $u_0\not\equiv 0$ in $\Omega$.\\
 \proof
 %As $\{u_k\}$ is bounded in $X_0$, so up to a subsequence $\|u_k\|\rightarrow {\rho_0\geq 0}$. Moreover, condition $J^{\prime}(u_k)\rightarrow 0$ and Lemma \ref{lem714} implies that
 We argue by contradiction. Suppose that $u_0\equiv 0$. Then using Lemma \ref{wk-sol}, we obtain
  \begin{equation}\label{mt-nw1}
  \int_\Om \left(\int_{\Om}\frac{G(y,u_k)}{|x-y|^{\mu}} dy \right) G(x,u_k)~dx\ra 0 \;\mbox{as}\; k\ra\infty.
  \end{equation}
  This together with $\ds\lim_{k\to\infty}J(u_k)= c_*$ gives that
  \[\lim_{k  \to \infty}\frac{s}{n} \hat{M}(\|u_k\|^{\frac{n}{s}}) = c_*< \frac{s}{n} \hat{M}\left(\left(\frac{2n-\mu}{2n} \al_{n,s}\right)^{\frac{n-s}{s}}\right).\]
  Thus $\hat M$ being increasing function gives that there exists a $k_0\in \mb N$  such that $\|u_k\|^{\frac{n}{s}}\leq \left(\frac{2n-\mu}{2n} \al_{n,s}\right)^{\frac{n-s}{s}}$ for all $k \geq k_0$.
{We fix $k \geq k_0$ and choose $p>1$ close to $1$ and $\e>0$ small enough such that
 \[\frac{2np(1+\e)}{2n-\mu}\|u_k\|^{\frac{n}{n-s}} < \alpha_{n,s}.\]
Using the growth assumptions on $g$ and Theorem \ref{moser} we have
\begin{align*}
 \|g(\cdot,u_k)u_k\|^{\frac{2n-\mu}{2n}}_{L^{\frac{2n}{2n-\mu}}(\Om)}&\leq  C(\e)\left(\int_\Om |u_k|^{\frac{2n\alpha}{2n-\mu}}dx + \int_\Om|u_k|^{\frac{2nr}{2n-\mu}}\exp\left(\frac{2n(1+\e)}{2n-\mu}|u_k|^{\frac{n}{n-s}}\right)dx \right)\\
     &\leq C(\e)\left(\int_\Om |u_k|^{\frac{2n\alpha}{2n-\mu}}dx +  \left(\int_\Om|u_k|^{\frac{2nrp^\prime}{2n-\mu}} dx\right)^{\frac{1}{p^\prime}}\right.\\
     &\quad \quad \quad\left.\left(\int_\Om\exp\left(\frac{2np(1+\e)}{2n-\mu}\|u_k\|^{\frac{n}{n-s}}\left(\frac{|u_k|}{\|u_k\|}\right)^{\frac{n}{n-s}}\right)dx\right)^{\frac{1}{p}} \right)
\end{align*}
{where $1<\alpha<l+1$ and $1<r$.} Thus,
     \begin{align}\label{grwth}
  \|g(\cdot,u_k)u_k\|_{L^{\frac{2n}{2n-\mu}}(\Om)}   & \leq C(\e)\left( \|u_k\|_{L^{\frac{2n\alpha}{2n-\mu}}(\Om)}^{\frac{2n-\mu}{2n\alpha}}+ \|u_k\|_{L^{\frac{2nrp^\prime}{2n-\mu}}(\Om)}^{\frac{2n-\mu}{2nr}} \right)\to 0 \;\text{as}\; k \to \infty,
     \end{align}
where $p^\prime$ denotes the H\"{o}lder conjugate of $p$ and $C(\e)>0$ is a constant depending on $\e$ which may change value at each step.
From the semigroup property of the Riesz potential and  Hardy-Littlewood-Sobolev inequality  we get that
 \begin{align*}
 &\left| \int_\Om \left(\int_{\Om}\frac{G(y,u_k)}{|x-y|^{\mu}} dy \right) g(x,u_k)u_kdx\right|\\
 &\leq \left(\int_\Om \left(\int_{\Om}\frac{G(y,u_k)}{|x-y|^{\mu}} dy \right) G(x,u_k)~dx\right)^{\frac12} \left(\int_\Om \left(\int_\Om\frac{g(y,u_k)u_k}{|x-y|^{\mu}} dy \right) g(x,u_k)u_k~dx\right)^{\frac12}\\
 & \leq \left(\int_\Om \left(\int_{\Om}\frac{G(y,u_k)}{|x-y|^{\mu}} dy \right) G(x,u_k)~dx\right)^{\frac12}C_{n,\mu}\|g(\cdot,u_k)u_k\|_{L^{\frac{2n}{2n-\mu}}(\Om)}\to 0
 \end{align*}
  as $k \to \infty$ using \eqref{mt-nw1} and \eqref{grwth}.}
 This together with $\langle J^{\prime}(u_k),u_k\rangle =0$ implies that $M(\|u_k\|^{\frac{n}{s}})\|u_k\|^{\frac{n}{s}} \ra 0$. From $(M3)$, we deduce that $\|u_k\|\ra 0$. Furthermore, we obtain $\lim_{k\ra \infty} J(u_k)=0=c_*$, which is a contradiction to the fact that $c_*>0$. Hence, we must have $u_0\not\equiv 0$.

\noi \textbf{Claim 2:} $\ds M(\|u_0\|^{\frac{n}{s}})\|u_0\|^{\frac{n}{s}}\geq \int_\Om \left(\int_{\Om}\frac{G(y,u_0)}{|x-y|^{\mu}} dy \right) g(x,u_0)u_0 dx$.\\
\proof Suppose by contradiction that $M(\|u_0\|^{\frac{n}{s}})\|u_0\|^{\frac{n}{s}}< \int_\Om\left(\int_{\Om}\frac{G(y,u_0)}{|x-y|^{\mu}} dy\right) g(x,u_0)u_0 ~dx$. That is, $\langle J^{\prime}(u_0),u_0\rangle<0.$

\noi It is easy to see, using $(M2)$, that $M(t)t \geq M(1)t^\gamma$ when $t\in [0,1]$. So for $0<t< \frac{1}{\|u_0\|}$, using Lemma \ref{lem7.3} and Hardy-Littlewood-Sobolev inequality we have that
\begin{align*}
\langle J^\prime(t u_0),u_0 \rangle  &\geq M(t^{\frac{n}{s}}\|u_0\|^{\frac{n}{s}})t^{\frac{n}{s}-1}\|u_0\|^{\frac{n}{s}} - { \frac{2s}{\gamma n}}\int_\Om  \left(\int_\Om \frac{g(y,tu_0)tu_0(y)}{|x-y|^{\mu}}~dy \right)g(x,tu_0)u_0(x)~dx\\
 &\geq M(1)t^{\frac{\gamma n}{s}-1}\|u_0\|^{\frac{\gamma n}{s}} - \frac{C}{t}\left(\int_\Om |g(x,tu_0)tu_0|^{\frac{2n}{2n-\mu}}~dx\right)^{\frac{2n-\mu}{n}}.
\end{align*}
But from the growth assumptions on $g$ we already know that for $\e>0$, $ \alpha >\frac{\gamma n}{2s}$ and $r > \frac{\gamma n}{2s}$,
\begin{align*}
&\left(\int_\Om |g(x,tu_0)tu_0|^{\frac{2n}{2n-\mu}}~dx\right)^{\frac{2n-\mu}{n}}\\
&\leq C(\e) \left( \int_\Om |tu_0|^{\frac{2n\alpha}{2n-\mu}} +  \|tu_0\|^{\frac{2r n}{2n-\mu}} \left( \int_\Om\exp\left(\frac{4n(1+\e)\|tu_0\|^{\frac{n}{n-s}}}{2n-\mu}\left(\frac{|tu_0|}{\|tu_0\|}\right)^{\frac{n}{n-s}}\right)\right)^{\frac12} \right)^{\frac{2n-\mu}{n}}\\
& \leq C(\e) \left(\|tu_0\|^{2\alpha}+ \|tu_0\|^{2r} \right)
\end{align*}
by choosing $t< \displaystyle\left(\frac{(2n-\mu)\alpha_{n,s}}{4n(1+\e)\|u_0\|^{\frac{n}{n-s}}}\right)^{\frac{n-s}{n}}$ and using Trudinger-Moser inequality. Therefore for $t>0$ small enough as above, we obtain
\[\langle J^\prime(t u_0),u_0 \rangle \geq M(1)t^{\frac{\gamma n}{s}-1}\|u_0\|^{\frac{\gamma n}{s}} - C(\e) \left(t^{2\alpha-1}\|u_0\|^{2\alpha}+ t^{2r-1}\|u_0\|^{2r} \right) \]
 which suggests that $\langle J^\prime(tu_0),u_0 \rangle>0$ when $t$ is sufficiently small. Thus there exists a  $\sigma\in(0,1)$ such that $\langle J^{\prime}(\sigma u_0),u_0\rangle=0$ that is,  $\sigma u_0\in \mc N$. Thus from  Lemmas \ref{lem7.3}, \ref{3.7}  and {Remark \ref{rem-M}}, it follows that
 {\small\begin{align*}
 &c_*\leq b \leq J(\sigma u_0)=J(\sigma u_0)-\frac{s}{n \gamma }\langle J^{\prime}(\sigma u_0),\sigma u_0\rangle\\
 &=\frac{s}{n} \hat M(\|\sigma u_0\|^{\frac{n}{s}})-\frac{s M(\|\sigma u_0\|^{\frac{n}{s}})\|\sigma u_0\|^{\frac{n}{s}}}{n\gamma}+ \frac{s}{n\gamma}\int_\Om \left(\int_{\Om}\frac{G(y,\sigma u_0)}{|x-y|^{\mu}}dy\right)\left(g(x,\sigma u_0)\sigma u_0- \frac{n\gamma}{2s} G(x, \sigma u_0)\right)\\
 &<\frac{s}{n}\hat M(\| u_0\|^{\frac{n}{s}})-\frac{s}{n\gamma}M(\|u_0\|^{\frac{n}{s}})\| u_0\|^{\frac{n}{s}}+\frac{s}{n \gamma}\int_\Om  \left(\int_{\Om}\frac{G(y, u_0)}{|x-y|^{\mu}}dy\right)\left(g(x, u_0)u_0- \frac{n\gamma}{2s} G(x,u_0)\right)dx.
 \end{align*}}
\noi Also by lower semicontinuity of norm and Fatou's Lemma, we obtain
 \begin{align*}
 c_*\leq b&<  \liminf_{k\rightarrow\infty}\left(\frac{s}{n}\hat M(\|u_k\|^{\frac{n}{s}})-\frac{s}{n \gamma}M(\|u_k\|^{\frac{n}{s}})\| u_k\|^{\frac{n}{s}}\right)\\
 &\quad+\liminf_{k\rightarrow\infty}\frac{s}{n\gamma}\int_\Om  \left(\int_{\Om}\frac{G(y,u_k)}{|x-y|^{\mu}}dy\right)\left[g(x, u_k)u_k-\frac{n\gamma}{2s} G(x,u_k)\right]dx\\
 &\leq \lim_{k\rightarrow\infty}\left[J(u_k)-\frac{s}{n\gamma}\langle J^{\prime}(u_k),u_k\rangle\right]=c_*,
 \end{align*}
 which is a contradiction. Hence Claim 2 is proved.\\
\noi \textbf{Claim 3:} $J(u_0)=c_*$.\\
\proof Using $\displaystyle \int_\Om \left(\int_{\Om}\frac{G(y,u_k)}{|x-y|^{\mu}}dy\right) G(x,u_k)dx \to \int_\Om \left(\int_{\Om}\frac{G(y,u_0)}{|x-y|^{\mu}}dy\right) G(x,u_0)dx$ and lower semicontinuity of norm we have $J(u_0)\leq c_*$.
 Now we are going to show that the case $J(u_0)<c_*$ can not occur.
 Indeed,  if $J(u_0)<c_*$ then $\|u_0\|^{\frac{n}{s}}<\rho_0^{\frac{n}{s}}.$ Moreover,
 \begin{equation}\label{7a3}
   \frac{s}{n}\hat M(\rho_0^{\frac{n}{s}})=\lim_{k\rightarrow\infty}\frac{s}{n}\hat M(\|u_k\|^{\frac{n}{s}})=c_*+\frac{1}{2}\int_\Om\left(\int_{\Om}\frac{G(y,u_0)}{|x-y|^{\mu}}dy \right)  G(x,u_0)dx,
 \end{equation}
  %$\frac{1}{n}\hat M(\rho_0^n)=\lim_{k\rightarrow\infty}\frac{1}{n}\hat M(\|u_k\|^{\frac{n}{s}})=c_*+\int G(x,u_0)dx$\\
  This gives that
  \[\ds \rho_0^{\frac{n}{s}} = \hat{M}^{-1}\left( {\frac{n}{s}} c_*+ {\frac{n}{2s}}\int_\Om \left(\int_{\Om}\frac{G(y,u_0)}{|x-y|^{\mu}}dy \right) G(x,u_0)dx\right).\]
  Next defining $v_k=\frac{u_k}{\|u_k\|}$ and $v_0=\frac{u_0}{\rho_0}$, we have $v_k\rightharpoonup v_0$ in $X_0$ and $\|v_0\|<1$. Thus by Lemma \ref{plc},
\begin{equation}\label{7a4}
 \sup_{k\in \mb N}\int_\Omega \exp({p|v_k|^{\frac{n}{n-s}}})~dx<\infty\;\;  \text{for all}\;  1<p<\frac{\alpha_{n,s}}{(1-\|v_0\|^{\frac{n}{s}})^\frac{s}{n-s}}.
 \end{equation}
\noi On the other hand,  by Claim 2, \eqref{cnd-M} and Lemma \ref{lem7.3}, we have
\begin{align*}
J(u_0) &\ge \frac{s}{n}\hat M(\|u_0\|^{\frac{n}{s}})-\frac{s}{n \gamma}M(\|u_0\|^{\frac{n}{s}})\|u_0\|^{\frac{n}{s}} \\
&\quad \quad+\frac{s}{n\gamma}\int_\Om \left(\int_{\Om}\frac{G(y,u_0)}{|x-y|^{\mu}}dy \right) \left(g(x,u_0)u_0- \frac{n\gamma}{2s} G(x,u_0)\right)dx \geq 0.
\end{align*}
 Using this together with Lemma \ref{lem7.2} and the equality,
$\frac ns\left(c_*-J(u_0)\right)=\hat M\left(\rho_{0}^{\frac{n}{s}}\right)-\hat M\left(\|u_0\|^{\frac{n}{s}}\right)$
we obtain
\[\hat M\left(\rho_{0}^{\frac{n}{s}}\right) \le \frac ns c_*+\hat M(\|u_0\|^{\frac{n}{s}})<\hat M\left(\left(\frac{2n-\mu}{2n}\al_{n,s}\right)^{\frac{n-s}{s}}\right)+\hat M(\|u_0\|^{\frac{n}{s}})\]
and therefore by $(M1)$
\begin{equation}\label{7a6}
\rho_{0}^{\frac{n}{s}}< \hat M^{-1}\left(\hat M\left(\left(\frac{2n-\mu}{2n}\al_{n,s}\right)^{\frac{n-s}{s}}\right)+\hat M(\|u_0\|^{\frac{n}{s}})\right)\le \left(\frac{2n-\mu}{2n}\al_{n,s}\right)^{\frac{n-s}{s}}+\|u_0\|^{\frac{n}{s}}.
\end{equation}
Since $\rho_{0}^{\frac{n}{s}}(1-\|v_0\|^{\frac{n}{s}})=\rho_{0}^{\frac{n}{s}}-\|u_0\|^{\frac{n}{s}}$, from \eqref{7a6} it follows that
\[ \rho_{0}^{\frac{n}{s}}< \frac{\left(\frac{2n-\mu}{2n}\al_{n,s}\right)^{\frac{n-s}{s}}}{1-\|v_0\|^{\frac{n}{s}}}.\]
Thus, there exists $\ba>0$ such that $ \|u_k\|^{\frac{n}{n-s}} < \ba < \frac{\al_{n,s}(2n-\mu)}
{2n(1-\|v_0\|^{\frac{n}{s}})^{\frac{s}{n-s}}}$ for $k$ large. We can choose $q>1$ close to $1$ such that $q \|u_k\|^{\frac{n}{n-s}} \le  \ba < \frac{(2n-\mu)\al_{n,s}}{2n(1-\|v_0\|^{\frac{n}{s}})^{\frac{s}{n-s}}}$ and using \eqref{7a4}, we conclude that for $k$ large
\[ \int_\Om \exp\left(\frac{2nq  |u_k|^{n/n-s}}{2n-\mu}\right) dx \le \int_\Om \exp\left(\frac{2n\beta |v_k|^{n/n-s}}{2n-\mu}\right)dx \le C.\]
Let us recall \eqref{rem2} and \eqref{grwth} to get that
\begin{align*}\left|\int_\Om \left(\int_{\Om}\frac{G(y,u_k)}{|x-y|^{\mu}}dy \right) g(x,u_k)u_k~dx\right| &\leq C\left( \|u_k\|_{L^{\frac{2n\alpha}{2n-\mu}}(\Om)}^{\frac{2n-\mu}{2n\alpha}}+ \|u_k\|_{L^{\frac{2nrq^\prime}{2n-\mu}}(\Om)}^{\frac{2n-\mu}{2nr}} \right)\\
& \to C\left( \|u_0\|_{L^{\frac{2n\alpha}{2n-\mu}}(\Om)}^{\frac{2n-\mu}{2n\alpha}}+ \|u_0\|_{L^{\frac{2nrq^\prime}{2n-\mu}}(\Om)}^{\frac{2n-\mu}{2nr}} \right)
\end{align*}
as $k \to \infty$. Then the pointwise convergence of $\left(\displaystyle\int_{\Om}\frac{G(y,u_k)}{|x-y|^{\mu}}dy \right) g(x,u_k)u_k$ to\\ $\left(\displaystyle\int_{\Om}\frac{G(y,u_0)}{|x-y|^{\mu}}dy \right) g(x,u_0)u_0$  as $k \to \infty$ asserts that
\[\lim_{k\to \infty}\int_\Om \left(\int_{\Om}\frac{G(y,u_k)}{|x-y|^{\mu}}dy \right) g(x,u_k)u_k~dx = \int_\Om \left(\int_{\Om}\frac{G(y,u_0)}{|x-y|^{\mu}}dy \right) g(x,u_0)u_0~dx\]
while using the Lebesgue dominated convergence theorem. Now Lemma \ref{PS-ws}, we get
\[\displaystyle\int_\Om \left(\int_{\Om}\frac{G(y,u_k)}{|x-y|^{\mu}}dy \right) g(x,u_k) (u_k-u_0)dx \rightarrow 0\;\mbox{as}\; k\rightarrow \infty.\]
 Since $\langle J^\prime (u_k), u_k-u_0\rangle \rightarrow 0$, it follows that
{\small \begin{equation} \label{na7new}
M(\|u_k\|^{\frac{n}{s}}) \int_{\mb R^{2n}} \frac{|u_k(x)-u_k(y)|^{\frac{n}{s}-2}(u_k(x)-u_k(y))((u_k-u_0)(x)-(u_k-u_0)(y))}{|x-y|^{2n}}dxdy \rightarrow 0.
\end{equation}}
\noi We define $U_{k}(x,y)=u_{k}(x)-u_{k}(y)$ and $U_{0}(x,y)=u_{0}(x)-u_{0}(y)$ then using $u_k\rightharpoonup u_0$ weakly in $X_0$ and boundedness of $M(\|u_k\|^{\frac{n}{s}})$, we have
{\small\begin{equation}\label{na8new}
M(\|u_k\|^{\frac{n}{s}}) \int_{\mb R^{2n}} \frac{|U_0(x,y)|^{\frac{n}{s}-2}  U_0(x,y) (U_k(x,y) -U_0(x,y))}{|x-y|^{2n}}dxdy\ra 0\; \mbox{as}\; k\ra\infty.
\end{equation}}
Subtracting \eqref{na8new} from \eqref{na7new}, we get
\[M(\|u_k\|^{\frac{n}{s}}) \int_{\mb R^{2n}} \frac{(|U_{k}(x,y)|^{\frac{n}{s}-2} U_{k}(x,y) - |U_{0}(x,y)|^{\frac{n}{s}-2} U_{0}(x,y)) (U_{k}(x,y)-U_{0}(x,y))}{|x-y|^{2n}}dxdy \rightarrow 0 \]
as $k\rightarrow \infty$. Now using this and the following
inequality
\begin{align}\label{11e2}
|a-b|^p\leq 2^{p-2}(|a|^{p-2}a-|b|^{p-2}b)(a-b)\;\mbox{for all}\; a,b\in \mb R\;\mbox{and}\; p\geq 2,
\end{align}
 with $a=u_k(x)-u_k(y)$ and $b=u_0(x)-u_0(y)$, we obtain
\[M(\rho_0^{\frac{n}{s}})\int_{\mb R^{2n}}\frac{|U_k(x)- U_0(x)|^{\frac{n}{s}}}{|x-y|^{2n}} dxdy \ra 0\;\mbox{as}\; k\ra\infty.\]
This implies that $u_k\ra u$ strongly in $X_0$ and hence $J(u)=c_*$ which is a contradiction.
Therefore, claim $3$ holds true.
%\noi Now by claim $3$ and \eqref{7a3} we can see that $\hat M(\rho_{0}^{\frac{n}{s}})=\hat M(\|u_0\|^{\frac{n}{s}})$ which shows that $\rho_{0}^{\frac{n}{s}}=\|u_0\|^{\frac{n}{s}}$.
Hence $J(u)=c_*= \lim\limits_{k\to \infty}J(u_k)$ and $\|u_k\| \to \rho_0$ gives that $\rho_0= \|u_0\|$. Finally we have
\begin{align*}
&M(\|u_0\|^{\frac{n}{s}}) \int_\Om \frac{|u_0(x)-u_0(y)|^{\frac{n}{s}-2}(u_0(x)-u_0(y))(\phi(x)-\phi(y))}{|x-y|^{2n}} dx dy\\
&\quad \quad=\int_\Om \left(\int_{\Om}\frac{G(y,u_k)}{|x-y|^{\mu}}dy \right)g(x,u_0) \phi ~dx,
\end{align*}
for all $\phi\in X_0$. Thus, $u_0$ is a non trivial  solution of $(\mc M)$. By Lemma \ref{pos-sol} we obtain that $u_0$ is the required nonnegative solution of $(\mc M)$ which completes the proof. \QED

\noindent{\bf Acknowledgements:} This research is supported by Science and Engineering Research Board, Department of Science and Technology, Government of India, Grant number:\\
ECR/2017/002651. The second author wants to thank Bennett University for its hospitality during her visit there.

%%%%%%%%%%%%%%%%%%%%%%%%%%%%%%%%%%%%%%%%%%%%%%%%%%%

\bibliographystyle{plain}
\bibliography{mybibfile}

%\bibitem{Moser}  J. Moser, A sharp form of an inequality  by N. Trudinger, Indiana Univ. Math. J. 20 (1971) no. 11, 1077-1092.
%\begin{thebibliography}{ll}

%\bibitem{lieb} E. Lieb and M. Loss, {\sl Analysis},  Graduate Studies in Mathematics, AMS, Providence, Rhode island, 2001.

%\item{ming} X. Mingqi, V. D. Radulescu and Binlin Zhang, {\it Fractionl Kirchhoff problems with critical Trudinger-Moser nonlinearity}, Calc. Var., (2019)
%\bibitem{Brezis-book} H. Brezis, Functional Analysis, Sobolev Spaces and Partial Differential Equations

% \kanishka perara-bifurcation results for fractional trudinger moser inequality
%\end{thebibliography}
\end{document}